\documentclass{amsart}

\usepackage{amsthm,amssymb,amsmath,a4wide,graphicx,tikz,bm,bbm,hyperref}
\usepackage[english]{babel}
\usepackage{xcolor}
\usepackage{subfigure}
\usepackage[normalem]{ulem}
\numberwithin{equation}{section}
\definecolor{lime}{HTML}{A6CE39}
\newtheorem{prop}{Proposition}[section]
\newtheorem{theorem}{Theorem}[section]
\newtheorem{lemma}{Lemma}[section]

\newtheorem{remark}{Remark}[section]
\theoremstyle{definition}

\newtheorem{ex}{Example}[section]
\definecolor{awesome}{rgb}{1.0, 0.13, 0.32}

\title[Birkhoff spectrum for self-affine sets and digit frequencies for GLS systems]{Birkhoff spectrum for diagonally self-affine sets and digit frequencies for GLS systems with redundancy}

\author[J. Imbierski]{Jonny Imbierski}
\address[J. Imbierski]{Mathematical Institute, Leiden University, PO Box 9512, 2300 RA Leiden, The Netherlands}
\email{imbierskijf@math.leidenuniv.nl}

\author[C. Kalle]{Charlene Kalle}
\address[C. Kalle]{Mathematical Institute, Leiden University, PO Box 9512, 2300 RA Leiden, The Netherlands}
\email{kallecccj@math.leidenuniv.nl}

\author[R. Mohammadpour]{Reza Mohammadpour}%\orcidA{}}
\address[R. Mohammadpour]{Department of Mathematics, Uppsala University, Box 480, SE-75106, Uppsala, Sweden}
\email{reza.mohammadpour@math.uu.se}

\date{Version of \today}

\begin{document}

\begin{abstract}
In this article, we calculate the Birkhoff spectrum in terms of the Hausdorff dimension of level sets for Birkhoff averages of continuous potentials for a certain family of diagonally affine IFS's. Also, we study Besicovitch-Eggleston sets for finite GLS number systems with redundancy. The redundancy refers to the fact that each number $x \in [0,1]$ has uncountably many expansions in the system. We determine the Hausdorff dimension of digit frequency sets for such expansions along fibres.
\end{abstract}


\maketitle

\section{Introduction}
%Carpet sets are subsets of the plane that are limit sets of iterated function systems (IFS) of contractions with linear parts given by diagonal matrices.
An {\em iterated function system} (IFS) is a finite collection of contracting maps on a complete metric space.
In this article, we consider IFSs $\{A_u+{\bm v}_u\}_{u\in\mathcal U}$ with $\mathcal U$ a finite set that are given by a collection $(A_u)_{u \in \mathcal U} \in GL_2(\mathbb R)^{\# \mathcal U}$
of matrices of the form
\begin{equation}\label{diag-matrix}
A_u=\left[\begin{array}{cc}
b_u & 0 \\
0 & c_u
\end{array}\right], \quad \text{with }\, 0 <|b_u|, |c_u| <1,  \, \, u \in \mathcal U,
\end{equation}
and a collection $(\bm{v}_u)_{u \in \mathcal U}$ of vectors of the form
\[ \bm{v}_u = \left[\begin{array}{cc}
\beta_u\\
\gamma_u
\end{array}\right] \in \mathbb R^2, \quad u \in \mathcal U.\]

% Let $\mathcal U$ be a finite set, $(A_u)_{u \in \mathcal U} \in GL_2(\mathbb R)^{\# \mathcal U}$ a collection 

The unique non-empty compact subset $\Lambda \subseteq [0,1]^2$ that satisfies
\begin{equation}\label{q:carpet}
\Lambda = \bigcup_{u \in \mathcal U} (A_u + \bm{v}_u)(\Lambda),
\end{equation}
which exists by Hutchinson's theorem \cite{Hutchinson}, is called the {\em attractor} or {\em self-affine set} of the  diagonally affine IFS $\{ A_u + \bm{v}_u\}_{u \in \mathcal U}$. 

\medskip
We are interested in multifractal properties of certain types of self-affine sets. Let $\sigma: \mathcal U^\mathbb N \to \mathcal U^\mathbb N$ denote the left shift, so $\sigma (\xi_n)_{n \ge 1} = (\xi_{n+1})_{n \ge 1}$.  For any continuous potential $\Phi: \mathcal U^\mathbb N \rightarrow \mathbb R^d$, $d \ge 1$, and for a given  vector $\bm{\alpha} = (\alpha_1, \ldots, \alpha_d) \in \mathbb R^d$, the symbolic level set is given by
\begin{equation}\label{q:symblevel} E_{\Phi}(\bm{\alpha}):=\left\{ \xi \in \mathcal U^\mathbb N \, :\, \lim _{n \rightarrow \infty} \frac{1}{n} \sum_{i=0}^{n-1} \Phi (\sigma^i(\xi))=\bm{\alpha}\right\}.
\end{equation}
For $\xi \in \mathcal U^\mathbb N$ and $n \ge 1$, let $A_{\xi|_{n-1}} = A_{\xi_1} \cdots A_{\xi_{n-1}}$ and define $\pi: \mathcal U^{\mathbb N} \to \Lambda $ by
\begin{equation}\label{q:pi}
\pi(\xi) = \sum_{n \ge 1} A_{\xi|_{n-1}} \bm{v}_{\xi_n}.
\end{equation}
The $\bm{\alpha}$-{\em level set} for $\Phi$ on $\Lambda$ is the set $\pi(E_{\Phi}(\bm{\alpha}))$. There are various known results on the size of level sets, both in terms of Hausdorff dimension and topological entropy (in the sense of \cite{Bowen}), see for instance \cite{BS, FFW, BSS, BSS1, Barrel, KR,  Ree11, Mohammadpour-Lyapunov, FH, Mohmmadpour24}. Multifractal results for self-affine sets in terms of Lyapunov dimensions were obtained in \cite{BJKR21} for  collections of matrices $(A_u)_{u \in \mathcal U}$ under certain strong-irreducibility and proximality conditions. One can find more information about the multifractal formalism in \cite{barreira_gelfert, climenhaga, Mohammadpour-survey}. Here, we consider the Hausdorff dimension of level sets for self-affine sets. 

\medskip
 Let $\mathcal F = \{ f_u: \mathbb R \to \mathbb R \}_{u \in \mathcal U}$ be an IFS of real-valued affine maps of the form $f_u(y) = d_u y + \delta_u$, so $|d_u| < 1$ for each $u \in \mathcal U$. For a sequence $\bm{u}=u_1\cdots u_n \in \mathcal U^n$, $n \ge 1$, we write
\[ f_{\bm{u}} (y)  := f_{u_1} \circ \cdots \circ f_{u_n} (y) = d_{\bm{u}}y + \delta_{\bm{u}}.\]
We say that $\mathcal F$ has {\em exact overlaps} if there are $\bm{u}, \bm{u}' \in \mathcal U^n$ for some $n \ge 1$ such that $f_{\bm{u}}= f_{\bm{u}'}$. For $\bm{u},\bm{u}'  \in \mathcal U^n$, we define the distance
\[\text{dist}(f_{\bm{u}}, f_{\bm{u}'}):=\left\{
  \begin{array}{@{}ll@{}}
    |\delta_{\bm{u}} -\delta_{\bm{u}'}|, & \text{if}\  d_{\bm{u}}= d_{\bm{u}'}; \\
    \infty, & \text{otherwise}.
  \end{array}\right.\]
In his breakthrough result, Hochman \cite{Hochman14} introduced the {\em Exponential Separation Condition} (ESC) to calculate the dimension of self-similar measures. We say that $\mathcal F$ satisfies the Exponential Separation Condition if there exists a constant $c>0$ and infinitely many integers $n \ge 1$ such that for all $\bm{u}, \bm{u}' \in \mathcal U^n$, 
\[ \text{dist}(f_{\bm{u}}, f_{\bm{u}'}) \ge c^n.\]

\medskip
We say that a diagonally affine IFS $\{A_u+{\bm v}_u\}_{u\in\mathcal U}$
 satisfies the {\em Strong Open Set Condition} (SOSC) if there is an open set $V \subseteq \mathbb R^2$ such that all the sets $(A_u + \bm{v}_u)(V)$ are pairwise disjoint, $\bigcup_{u \in \mathcal U} (A_u + \bm{v}_u) (V) \subseteq V$ and $\Lambda \cap V \neq \emptyset$, where $\Lambda$ is as in \eqref{q:carpet}. \medskip

Our first result is for the following class of diagonally affine IFSs. Let $\mathcal D$ be the collection of all IFSs $\{A_u+ \bm{v}_u\}_{u\in \mathcal U}$ with matrices as in \eqref{diag-matrix} that satisfy the SOSC together with either (D) or (D'):
\begin{itemize}
\item[(D)] $|b_u|>|c_u|$ for all $u \in \mathcal U$ and
\begin{itemize}
\item[(a)] the IFS obtained from projecting to the first coordinate $\mathcal G_1:= \{ g_{1,u}(y)= b_u y +\beta_u\}_{u \in \mathcal U}$ satisfies the ESC, or,
\item[(b)] $b_u$ is algebraic for all $u \in \mathcal U$ and the IFS obtained from projecting to the first coordinate $\mathcal G_1:= \{ g_{1,u}(y)= b_u y +\beta_u\}_{u \in \mathcal U}$ has no exact overlaps;
\end{itemize}
\item[(D')] $|b_u|<|c_u|$ for all $u \in \mathcal U$ and
\begin{itemize}
\item[(a)] the IFS obtained from projecting to the second coordinate $\mathcal G_2:= \{ g_{2,u}(y)= c_u y +\gamma_u\}_{u \in \mathcal U}$ satisfies the ESC, or,
\item[(b)]  $c_u$ is algebraic for all $u \in \mathcal U$ and the IFS obtained from projecting to the second coordinate $\mathcal G_2:= \{ g_{2,u}(y)= c_u y +\gamma_u\}_{u \in \mathcal U}$ has no exact overlaps.
\end{itemize}
\end{itemize}
An IFS satisfying condition (D')(a) is shown in Figure~\ref{f:2and3}(a).

\medskip
\noindent Let $\mathcal{M}(\mathcal U^\mathbb N, \sigma)$ denote the set of all $\sigma$-invariant Borel probability measures on $\mathcal U^\mathbb N$ and let
\begin{equation}\label{q:Lphi}
\begin{aligned}
L_{\Phi} & :=\left\{\bm{\alpha} \in \mathbb{R}^d\, :\, \exists\,   \xi \in \mathcal U^\mathbb N  \text { with } \lim _{n \rightarrow \infty} \frac{1}{n} \sum_{i=0}^{n-1} \Phi (\sigma^i \xi)=\bm{\alpha} \right\} \\
& =\left\{\bm{\alpha} \in \mathbb{R}^d\, :\, \exists \,  \mu \in \mathcal{M}(\mathcal U^\mathbb N, \sigma) \text { with } \int_{\mathcal U^\mathbb N} \Phi \, \mathrm{d} \mu=\bm{\alpha} \right\}
\end{aligned}
\end{equation}
%ergodic decomposition + empirical measure
be the collection of all vectors $\bm{\alpha}$ for which the corresponding level set is non-empty, known as the {\em spectrum} of $\Phi$. We use $\mathring{L}_{\Phi}$ to denote the interior of $L_\Phi$. For any $\mu \in \mathcal{M}(\mathcal U^\mathbb N, \sigma)$, let $\operatorname{dim}_{\mathrm{L}}(\mu)$ denote the Lyapunov dimension of $\mu$. Let $P$ denote the topological pressure. Our first result is as follows.

\begin{theorem}\label{t:main1}
Let $\left\{A_u+ \bm{v}_u\right\}_{u \in \mathcal U} \in \mathcal D$ and let $\Phi: \mathcal U^\mathbb N \rightarrow \mathbb{R}^d$, $d \ge 1$, be a continuous potential. Then for each $\bm{\alpha} \in \mathring{L}_{\Phi}$,
\[ \begin{aligned}
\operatorname{dim}_{\mathrm{H}}\left(\pi(E_{\Phi}(\bm{\alpha}))\right) & =\sup \bigg\{\operatorname{dim}_{\mathrm{L}}(\mu)\, :\,  \mu \in \mathcal{M}(\mathcal U^\mathbb N, \sigma) \text { and } \int_{\mathcal U^\mathbb N} \Phi\, \mathrm{d} \mu=\bm{\alpha} \bigg\} \\
& =\sup \bigg\{s \geq 0\, :\,  \inf _{q \in \mathbb{R}^d} P\left(\log \varphi^s+\langle q, \Phi-\bm{\alpha}\rangle\right) \geq 0\bigg\}.
\end{aligned}\]
\end{theorem}

\medskip
The second family of IFS sets that we consider in this article is motivated by a specific type of representations of real numbers called {\em Generalised L\"uroth Series (GLS) expansions} as described in \cite{BBDK94}. As the name suggests, GLS expansions are generalisations of L\"uroth expansions, which were introduced in 1883 by L\"uroth \cite{Lur83} and for $x \in [0,1]$ have the form
\[ x = \sum_{n \ge 1} \frac{a_n}{\prod_{i=1}^n a_i(a_i+1)}, \quad a_n \in \mathbb N, \, n \ge 1.\]
L\"uroth expansions can be obtained from the IFS $\{ l_k:[0,1]\to [0,1]\}_{k \in \mathbb N}$ where $l_k(x) = \frac{k+x}{k(k+1)}$, see e.g.~\cite{JV69}. If for $x \in [0,1]$ there is a sequence $(a_n)_{n \ge 1} \in \mathbb N^\mathbb N$ such that
\[ x = \lim_{n \to \infty} l_{a_1} \circ \cdots \circ l_{a_n}(0),\]
then $x$ has a L\"uroth expansion with digits given by $(a_n)_{n \ge 1} \in \mathbb N^\mathbb N$. While L\"uroth expansions take their digits from the infinite digit set $\mathbb N$ and all terms in the expansion are positive, a GLS number system can have either finite or infinite digit sets and the corresponding GLS expansions can have both positive and negative terms. Given a finite or countably infinite digit set $\mathcal I$, a partition $\{ [\ell_k,r_k] \}_{k\in \mathcal I}$ of $[0,1]$ into closed intervals and a vector $ (\varepsilon_k)_{k\in \mathcal I} \in \{0,1\}^{\# \mathcal I}$, one can consider the IFS
\begin{equation}\label{q:singlegls}
\{ g_k:[0,1]\to [0,1]\}_{k\in \mathcal I},
\end{equation}
where $g_k$ maps the interval $[0,1]$ affinely onto $[\ell_k,r_k]$ in an orientation-preserving manner if $\varepsilon_k=0$ and in an orientation-reversing manner if $\varepsilon_k=1$. In other words, if we write $K_k = (r_k-\ell_k)^{-1}$, then $g_k(x) = \ell_k + \frac{\varepsilon_k +  (-1)^{\varepsilon_k}x }{K_k}$. Since $g_k([0,1]) = [\ell_k,r_k]$ for each $k$, it follows that for each $x \in [0,1]$, there is a sequence $(a_n)_{n \ge 1}$ such that
\[ x = \lim_{n \to \infty} g_{a_1} \circ \cdots \circ g_{a_n}(0).\]
Thus, $x$ can be expressed as
\begin{equation}\label{q:genexp}
x = \sum_{n \ge 1} (-1)^{\sum_{i=1}^{n-1}\varepsilon_{a_i}} \frac{\ell_{a_n}K_{a_n}+\varepsilon_{a_n}}{\prod_{i=1}^n K_{a_i}},
\end{equation}
which is called a {\em GLS expansion} of $x$ with digit set $\mathcal I$. Here, we let $\sum_{i=1}^0 \varepsilon_{a_i}=0$ and $\prod_{i=1}^0K_{a_i}=1$. One recovers the L\"uroth expansions by taking $\mathcal I=\mathbb N$, $[\ell_k,r_k] = [\frac1{k+1}, \frac1k]$ and $\varepsilon_k=0$ for each $k \ge 1$ and one obtains integer base $N$-expansions by setting $\mathcal I = \{0,1, \ldots, N-1\}$ and taking $[\ell_k,r_k] = [\frac{k}{N}, \frac{k+1}{N}]$ and $\varepsilon_k=0$. The expansions from \eqref{q:genexp} can also be seen as signed versions of Cantor base expansions, as introduced by Cantor in \cite{Can69}. GLS expansions have been considered previously in \cite{Mun11,Arr15,KKSS15,JMS18,HK23,BK24} and recently also in relation to neural networks \cite{BKSN19,RHN23}. Level sets for L\"uroth expansions and more generally GLS expansions have been considered in particular with respect to digit frequencies, see \cite{BI09}. Such level sets are known as {\em Besicovitch-Eggleston sets} due to the results from \cite{Bes35} by Besicovitch and \cite{Egg49} by Eggleston on the Hausdorff dimension on digit frequency level sets for integer base expansions.

\medskip
In the above setting, for any given GLS number system, all but countably many numbers in $[0,1]$ have a unique GLS expansion in that system and the numbers that do not have a unique expansion have exactly two expansions. In this article, we consider IFSs that correspond to GLS number systems with redundancy, that is, in which all numbers have uncountably many different representations in the system. Number systems with redundancy have proven interesting in several settings, including signed binary expansions where they are used to find so-called minimal weight expansions, i.e.\! expansions that maximise the number of digits 0, see e.g.~\cite{MO90,KT93,DK20}, and in non-integer base expansions in relation to applications in analogue-to-digital converters and random number generation, see e.g.~\cite{DDGV06,JM16}. Number systems with redundancy have also been considered in \cite{KKV17,KMTV22,KM22b} for continued fraction expansions and L\"uroth expansions. To obtain a GLS number system with redundancy, we let $J \in \mathbb N_{\ge 2}$ and start with $J$ IFSs that correspond to $J$ different GLS number systems with finite digit sets. We combine these into one diagonally affine IFS on $\mathbb R^2$, which we call a GLS IFS, by using a positive probability vector $ (p_j)_{0 \le j < J}$, so $p_j >0$ for all $0 \le j < J$ and $\sum_{0 \le j< J} p_j=1$. This vector $(p_j)_{0 \le j < J}$ can be thought of as the probabilities with which the $j$-th GLS number system is chosen to generate the $n$-th digit in the expansions for any $n \ge 1$. Therefore, a GLS IFS is given by the following data:
\begin{itemize}
\item[1.] an integer $J \in \mathbb N_{\ge 2}$ and a positive probability vector $(p_j)_{0 \le j < J}$;
\item[2.] for each $0 \le j < J$, a number $B_j \in \mathbb N_{\ge 2}$, a partition $0=r_{(j,0)} < r_{(j,1)} < \cdots < r_{(j,B_j)}=1$ and a vector $(\varepsilon_{(j,k)})_{0 \le k < B_j} \in \{0,1\}^{B_j}$.\\
\end{itemize}
If we set $\mathcal E = \{ (j,k) \, : \, 0 \le j < J, \, 0 \le k < B_j\}$ and for each $(j,k) \in \mathcal E$ let $q_{(j,k)}=r_{(j,k+1)}-r_{(j,k)}$ and
\begin{equation}\label{q:Av}
A_{(j,k)} = \left[\begin{array}{cc}
\displaystyle (-1)^{\varepsilon_{(j,k)}}q_{(j,k)} & 0 \\
0 & p_j
\end{array}\right], \quad \bm{v}_{(j,k)} = \left[\begin{array}{c}
\displaystyle r_{(j,k)} + \varepsilon_{(j,k)}q_{(j,k)} \\
\sum_{i=0}^{j-1}p_i
\end{array} \right],
\end{equation}
then we call the IFS $\{ A_e + \bm{v}_e\}_{e \in \mathcal E}$ a {\em GLS IFS}. See Figure~\ref{f:2and3}(b) for an example.

\begin{center}
\begin{figure}[ht]
\subfigure[Example of an IFS satisfying (D')(a)]
{
\begin{tikzpicture}[scale=2]
\draw[white](-.5,.5)--(1.5,.5);
\draw[dotted](0,.6)--(1,.6)(2/3,0)--(2/3,1);
\filldraw[fill=green!40!blue!90!black, draw=black] (0,0) rectangle (1/2,.6);
\filldraw[fill=green!60!blue!90!black, draw=black] (2/3,.6) rectangle (1,1);
\draw[thick](0,0)--(1,0)--(1,1)--(0,1)--cycle;
\draw[thick](0,.6)node[left]{$p$}--(1/2,.6)(2/3,.6)--(1,.6);
\draw[thick](.5,0)node[below]{$\frac12$}--(1/2,.6)(2/3,.6)--(2/3,1);
\node[below] at (2/3,0){$\frac23$};
\end{tikzpicture}}
\hspace{1.5cm}
\subfigure[Example of a GLS IFS]
{\begin{tikzpicture}[scale=2]
\draw[white](-.5,.5)--(1.5,.5);
\filldraw[fill=green!40!blue!90!black, draw=black] (0,0) rectangle (1/2,.6);
\filldraw[fill=green!20!red!90!black, draw=black] (1/2,0) rectangle (1,.6);
\filldraw[fill=green!60!blue!90!black, draw=black] (2/3,.6) rectangle (1,1);
\filldraw[fill=blue!60!red!90!black, draw=black] (1/3,.6) rectangle (2/3,1);
\filldraw[fill=yellow!60!red!90!black, draw=black] (0,.6) rectangle (1/3,1);
\draw[thick](0,0)--(1,0)--(1,1)--(0,1)--cycle;
\draw[thick](0,.6)node[left]{$p$}--(1,.6);
\draw[thick](.5,0)node[below]{$\frac12$}--(1/2,.6)(1/3,.6)--(1/3,1)--(2/3,1);
\draw[dotted] (1/3,0)--(1/3,1)(2/3,0)--(2/3,1);
\node[below] at (1/3,0){$\frac13$};
\node[below] at (2/3,0){$\frac23$};
\end{tikzpicture}}
\caption{Two examples of IFSs. The coloured rectangles indicate the images of the unit square under the maps in the IFS.}
\label{f:2and3}
\end{figure}
\end{center}

% \medskip
We mention a few particular properties of GLS IFSs. Each GLS IFS satisfies the strong open set condition and has $\Lambda = [0,1]^2$. For the projection onto the first coordinate, we use $\mathcal G_1=\{ h_e:[0,1] \to [0,1]\}_{e \in \mathcal E}$, where for each $e \in \mathcal E$, we set
\[ h_e(x) = r_e + q_e(\varepsilon_e + (-1)^{\varepsilon_e} x).\]
Without additional assumptions, the GLS IFS need not fall into one of the categories (D)(a) or (D)(b). The projection on the second coordinate $\mathcal G_2 = \{ g_{2,(j,k)}(y)= p_j y + \sum_{i=0}^{j-1} p_i \}_{(j,k) \in \mathcal E}$ of a GLS IFS contains several duplicates of each map. Therefore, GLS IFSs do not fall into the class of diagonally affine IFSs that satisfy (D'), but by removing these duplicates they can potentially contain a subcollection of contractions that satisfies (D') as shown in Figure~\ref{f:2and3}.

\medskip
We can obtain number expansions from a GLS IFS in the following way. For each $x \in [0,1]$, there are sequences $(e_m)_{m \in \mathbb N} \in \mathcal E^\mathbb N$ such that $x$ can be written as
\begin{equation}\label{q:glsifs}
x = \lim_{m \to \infty} h_{e_1} \circ \cdots \circ h_{e_m} (0).
\end{equation}
If for each $m \in N$ we write
\begin{equation}\label{q:smKmtm} s_m = \varepsilon_{e_m}, \quad K_m = q_{e_m}^{-1}, \quad t_m = r_{e_m} + \varepsilon_{e_m}K_m^{-1},
\end{equation}
then it follows from \eqref{q:glsifs} that
\begin{equation}\label{q:fglse}
x = \sum_{m \ge 1} (-1)^{\sum_{i=1}^{m-1}s_i} \frac{t_m}{\prod_{i=1}^{m}K_i}
\end{equation}
and we can see the resemblance with L\"uroth expansions. It is shown in Proposition~\ref{p:uncountable} below that, under the additional assumption on the GLS IFS that $h_e \neq h_{e'}$ whenever $e \neq e'$, indeed all numbers $x \in [0,1]$ have uncountably many different representations of the form \eqref{q:glsifs}. We give several examples of GLS IFSs and the associated number expansions at the end of the article.

\medskip
For GLS IFSs, we consider the potential that captures digit frequencies. For $e \in \mathcal E$, let $[e] \subseteq \mathcal E^\mathbb N$ denote the cylinder set of those sequences that have $e \in \mathcal E$ as their first term and $\mathbbm 1_{[e]}: \mathcal E^\mathbb N \to \{0,1\}$ the indicator function on $[e]$. Define the continuous potential $\mathbbm 1: \mathcal E^\mathbb N \to \{0,1\}^{\#\mathcal E}$ by $\mathbbm 1(\omega) = (\mathbbm 1_{[e]}(\omega))_{e \in \mathcal E}$. For each $e \in \mathcal E$ and $\omega \in \mathcal E^\mathbb N$, write
\[ \tau_e (\omega) = \lim_{n \to \infty} \frac{\# \{1 \le i \le n \, : \, \omega_i = e\}}{n}\]
for the frequency of the digit $e$ in $\omega$. Consider a frequency vector $\bm{\alpha} = (\alpha_e)_{e \in\mathcal E} \in [0,1]^{\# \mathcal E}$, i.e.\! that satisfies $\sum_{e \in \mathcal E} \alpha_e =1$, and let $F(\bm{\alpha}) = \pi (E_{\mathbbm 1}(\bm{\alpha}))$. Then
\begin{equation}\label{q:Falpha}
F(\bm{\alpha}) = \{ (x,y) \in [0,1]^2 \, : \, \exists\, \omega\in\pi^{-1}\{(x,y)\}\, \textrm{ s.t. } \tau_e(\omega) = \alpha_e \, \text{ for all } \,  e \in \mathcal E \}
\end{equation}
is the {\em GLS digit frequency level set} or {\em Besicovitch-Eggleston set} for $\bm{\alpha}$. Results on $\dim_{\mathrm H} (F(\bm{\alpha}))$ have been obtained in \cite[Theorem 1]{Nie99} in the specific case of Bedford-McMullen carpets, that is with $p_j = \frac1{J}$ for each $0 \le j <J$ and $(-1)^{\varepsilon_e}=1$ and $q_e=\frac1N$ for some fixed $N \in \mathbb N_{\ge 2}$ and all $e \in \mathcal E$. This result was extended in \cite[Corollary 1]{Ree11} for Lalley-Gatzouras carpets, which are similar to our setting but have the additional requirements that $\varepsilon_{(j,k)} =0$ and $q_{(j,k)} \le p_j$ for all $(j,k) \in \mathcal E$. In the current setting, a lower bound for $\dim_{\mathrm H}(F(\bm{\alpha}))$ in terms of the Ledrappier-Young formula for the Bernoulli measure $\mu_{\bm{\alpha}}$ can be deduced from \cite[Theorem 2.3 and Corollary 2.8]{BARANY2017}  in case the two Lyapunov exponents of the Bernoulli measure of the system differ and the frequency vector $\bm{\alpha}$ is strictly positive. In case the two Lyapunov exponents  of the Bernoulli measure $\mu_{\bm{\alpha}}$ are equal, one can apply \cite{FengHu} to obtain a similar lower bound in terms of the Ledrappier-Young formula for $\mu_{\bm{\alpha}}$.

\medskip
Here, we will instead, for fixed $y \in [0,1]$, focus on the {\em fibre level sets} 
\begin{equation}\label{q:Fwalpha}
F_y({\bm{\alpha}}):= \left\{x\in[0,1]\, :\, \exists\,\omega\in\pi^{-1}\{(x,y)\}\, \textrm{ s.t. }\tau_e(\omega)=\alpha_e\,\, \text{ for all } e\in\mathcal E\right\}.
\end{equation}
We only consider frequency vectors $\bm{\alpha}$ with $\alpha_j:= \sum_{k=0}^{B_j-1} \alpha_{(j,k)} >0$ for all $0 \le j < J$ (otherwise we could just as well have considered a smaller GLS IFS).
%The point $w$ determines the sequences $(j_m)_{m \ge 1} \in \{0,1, \ldots, J-1\}^\mathbb N$ in the sequences $\omega = (j_m,k_m)_{m \ge 1} \in \pi^{-1}\{(w,x)\}$. So, $F_w({\bm\alpha})$ is the set of all points $x \in [0,1]$ that have a GLS expansion with digit sequence $\omega = (j_m, k_m)_{m \ge 1} \in \mathcal E^\mathbb N$ if
%\[ w = \lim_{m \to \infty} f_{j_1} \circ f_{j_2} \circ \cdots \circ f_{j_m}(0).\]
Let
\begin{equation}\label{q:walphadef}
W(\bm{\alpha}) := \{ y \in [0,1] \, : \, F_y(\bm{\alpha}) \neq \emptyset \}.
\end{equation}
Let $\mu_{\bm{\alpha}}$ be the $\bm{\alpha}$-Bernoulli measure on $\mathcal E^\mathbb N$. For each $0 \le j < J$, let $f_j:[0,1]\to [0,1]$ be the map given by
\begin{equation}\label{q:fj}
f_j(y) = p_jy + \sum_{i=0}^{j-1}p_i
\end{equation}
and define the map
\begin{equation}\label{q:pi1} \pi_2: \mathcal E^\mathbb N \to [0,1]; \, (j_m,k_m)_{m \ge 1} \mapsto \lim_{m \to \infty} f_{j_1} \circ \cdots \circ f_{j_m}(0).
\end{equation}
Set $\nu_{\bm{\alpha}} = \mu_{\bm{\alpha}} \circ \pi_2^{-1}$. As we will see later, $\nu_{\bm{\alpha}}(W(\bm{\alpha}))=1$. 
%Moreover, for each $0 \le j < J$ let $\alpha_j = \sum_{k=0}^{B_J -1} \alpha_{(j,k)}$ and set
%If we assume that the GLS IFS $\{A_e + \bm{v}_e\}_{e \in \mathcal E}$ and the frequency vector $\bm{\alpha} = (\alpha_e)_{e \in \mathcal E} \in (0,1)^{\# \mathcal E}$ together satisfy the following condition:
%\medskip
%\noindent {\bf(U)} \quad For each $(w,x) \in F(\bm{\alpha})$ it holds that $\#\pi^{-1}\{(w,x)\} =1$,
%\medskip
%{\color{red}That makes $\pi$ bijective, right?  By the way, the $\pi$ we used here is different from the $\pi$ introduced at the beginning. }
We have the following results on the Hausdorff dimension of the fibre Besicovitch-Eggleston sets.

\begin{theorem}\label{t:main4}
    Let $\{ A_e + \bm{v}_e\}_{e \in \mathcal E}$ be a GLS IFS and ${\bm\alpha}=(\alpha_e)_{e\in\mathcal E} \in [0,1]^{\# \mathcal E}$ a frequency vector. Then 
    $$\dim_{\mathrm{H}}(F_y({\bm\alpha}))
        \leq\frac{\sum_{e \in \mathcal E} \alpha_e \log \alpha_e - \sum_{0 \le j < J} \alpha_j \log \alpha_j}{\sum_{e\in\mathcal E} \alpha_e\log q_e}$$
    for all $y\in W({\bm\alpha})$.
    Furthermore, if $\bm{\alpha}$ satisfies that for each $0 \le j < J$ there are $k, \ell \in B_j$ with $k \neq \ell$, $\alpha_{(j,k)}>0$ and $\alpha_{(j, \ell)}>0$, then
    $$\dim_{\mathrm{H}}(F_y({\bm\alpha}))
        \geq\frac{\sum_{e \in \mathcal E} \alpha_e \log \alpha_e - \sum_{0 \le j < J} \alpha_j \log \alpha_j}{\sum_{e\in\mathcal E} \alpha_e\log q_e}$$
    for $\nu_{\bm\alpha}$-a.e.~$y\in W(\bm\alpha)$.
\end{theorem}

Fibrewise results similar in spirit to Theorem~\ref{t:main4} were obtained in \cite{NT}, where the authors study real numbers with a semi-regular continued fraction expansion that satisfies a certain growth condition on its digits.

\medskip
The paper is outlined as follows. In Section~\ref{s:preliminaries}, we provide the necessary preliminaries. We prove Theorem~\ref{t:main1} in Section~\ref{s:ddaifs}. Section~\ref{s:fgls} is devoted to GLS IFSs. Here, we show that given a GLS IFS that has $h_e \neq h_{e'}$ whenever $e \neq e'$, all $x \in [0,1]$ have uncountably many expansions of the form \eqref{q:glsifs}. We then continue with some results on the spectrum of the Besicovitch-Eggleston sets $F(\bm{\alpha})$ and on the sets $W(\bm{\alpha})$, which will be used in the proof of Theorem~\ref{t:main4}. This section also contains the proof of Theorem~\ref{t:main4}. % and Corollary~\ref{t:main2}.
Finally, Section~\ref{s:examples} contains some examples.
 
\subsection{Acknowledgments.}
The first two authors acknowledge the hospitality of Uppsala University and the support of the Knut and Alice Wallenberg Foundation for their research visits. The third author acknowledges the hospitality of Leiden University and the support of the Knut and Alice Wallenberg Foundation and the Swedish Research Council under grant no.~2021-06594 while the author was in residence at Institut Mittag-Leffler in Djursholm, Sweden during the program ``Two Dimensional Maps". We also thank the anonymous referee for their valuable corrections
and suggestions.

\section{Preliminaries}\label{s:preliminaries}
In this section, we introduce notation and collect several bits of information that are used for the results in the later sections. 

\subsection{Strings and sequences}
Let $\mathcal U$ be a finite set of symbols and denote by $\mathcal U^\mathbb N$ the set of one-sided infinite sequences of elements in $\mathcal U$. For each $n \ge 0$, the set $\mathcal U^n$ is the set of {\em words} of length $n$, where we let $\mathcal U^0 = \{ \varnothing \}$ be the set containing only the empty word, which we denote by $\varnothing$. Let $\mathcal U^* = \bigcup_{n \ge 0}\mathcal U^n$ be the set of all words. For a word $\bm{u} \in \mathcal U$, we use the notation $|\bm{u}|$ for its length, so $|\bm{u}| = n$ if $\bm{u} \in \mathcal U^n$. If $\bm{u} = u_1 \cdots u_n \in \mathcal U^*$, then for each $k \le n$, we use the notation $\bm{u}|_k = u_1 \cdots u_k$. Similarly for a sequence $\xi \in \mathcal U^\mathbb N$ and any $n \ge 1$, we set $\xi|_n = \xi_1 \cdots \xi_n$. %For $\bm{u} \in \mathcal U^*$ and $\xi \in \mathcal U^* \cup \mathcal U^\mathbb N$ their concatenation is denoted by $\bm{u} \xi$. %For $ \xi \in \mathcal U^* \cup \mathcal U^\mathbb N$ and $\upsilon \in \mathcal U^* \cup \mathcal U^\mathbb N$ we use $\xi \wedge \upsilon$ to denote their longest common prefix, i.e.\! $\xi \wedge \upsilon = \xi|_k$, where
%\[ k= \inf \{ i \ge 1 \, : \, \xi_i \neq \upsilon_i \}.\]
The {\em cylinder set} corresponding to a word $\bm{u} \in \mathcal U^n$, $n \ge 0$, is denoted by
\[ [\bm{u}] = \{ \xi \in \mathcal U^\mathbb N \, : \, \xi|_n = \bm{u} \}.\]
For any sequence $\xi \in \mathcal U^\mathbb N$, any symbol $u \in \mathcal U$ and any $n \in \mathbb N$, we use the notation
\[ \tau_u(\xi,n) = \# \{1 \le m \le n \, : \, \xi_m = u\} \]
for the number of times the symbol $u$ occurs in the first $n$ elements of $\xi$ and
\[ \tau_u(\xi) = \lim_{n \to \infty} \frac{\tau_u(\xi,n)}{n}\]
for the frequency of the digit $u$ in $\xi$ if it exists. We use this notation in Section~\ref{s:fgls}.

\medskip
We can equip $\mathcal U^\mathbb N$ with a metric $\eta$ to obtain a compact metric space by setting 
\[ \eta : \mathcal U^\mathbb N \to [0,1]; \,  (\xi,\upsilon) \mapsto \begin{cases}
2^{-\min \{ n \ge 1\, : \, \xi_n \neq \upsilon_n \}}, & \text{if } \xi \neq \upsilon,\\
0, & \text{if } \xi=\upsilon.
\end{cases}\]
The {\em left shift} is denoted by $\sigma: \mathcal U^\mathbb N \to \mathcal U^\mathbb N$, i.e, $\sigma(\xi)_n = \xi_{n+1}$ for each $n \ge 1$. With a slight abuse of notation, we will use $\sigma$ to denote the left shift on any sequence space without specifying the alphabet as a subscript whenever no confusion can arise. Cylinder sets are both open and closed and generate the Borel $\sigma$-algebra on $\mathcal U^\mathbb N$. Let $\mathcal M (\mathcal U^\mathbb N, \sigma)$ denote the set of all shift-invariant Borel probability measures on $\mathcal U^\mathbb N$.
For $\mu \in \mathcal M (\mathcal U^\mathbb N, \sigma)$, we use $h_{\mu}(\sigma)$ to denote the measure-theoretic entropy of $\mu$ with respect to $\sigma$, which is defined by
\[
h_{\mu}(\sigma):=-\lim_{n \to \infty}\frac{1}{n}\sum_{\bm{u} \in \mathcal{U}_n} \mu([\bm{u}])\log \mu([\bm{u}]),
\]
where $0\log 0=0.$
Given a probability vector $\bm{p} = (p_u)_{u \in \mathcal U}$, the $\bm{p}$-{\em Bernoulli measure} $\mu_{\bm{p}}$ is the probability measure on $(\mathcal U^\mathbb N, \sigma)$ that is defined on the cylinder $[\bm{u}]=[u_1 \cdots u_n]$ by
\[ \mu_{\bm{p}}([\bm{u}])=p_{u_1} \cdots p_{u_n}.\]
Moreover, the measure-theoretic entropy of $\mu_{\bm{p}}$ with respect to $\sigma$ is given by
\begin{equation}\label{q:entropy}
h_{\mu_{\bm{p}}}(\sigma) = - \sum_{u \in \mathcal U} \mu_{\bm{p}}([u]) \log (\mu_{\bm{p}}([u])).
\end{equation}

\subsection{Matrix products}\label{sec:MatrixProd}
Let $(A_u)_{u \in \mathcal U} \in GL_2(\mathbb R)^{\# \mathcal U}$ be a collection of matrices as in \eqref{diag-matrix}. Recall that for a sequence $\xi=(\xi_n)_{n \ge 1} \in \mathcal U^\mathbb N$ and $n \in \mathbb N$, we set $A_{\xi|_n} = A_{\xi_1} \cdots A_{\xi_n}$.  For the entries on the diagonal of $A_{\xi|_n}$, write $b_{\xi|_n} = b_{\xi_1} \cdots  b_{\xi_n}$ and $c_{\xi|_n}=c_{\xi_1} \cdots  c_{\xi_n}$.  For $\bm{u} = u_1 \cdots u_n \in \mathcal U^n$, we similarly write $A_{\bm{u}} = A_{u_1} \cdots A_{u_n}$ with $b_{\bm{u}} = b_{u_1} \cdots  b_{u_n}$ and $c_{\bm{u}}=c_{u_1} \cdots  c_{u_n}$ for the diagonal entries.

\medskip
Let $\mathbb{P}^1$ be the real projective line, which is the set of all lines through the origin in $\mathbb R^2$. We say that a proper subset $\mathcal{C} \subset \mathbb{P}^1$ is a \textit{cone} if it is a closed projective interval and a \textit{multicone} if it is a finite union of cones. The collection  $ (A_u)_{u \in \mathcal U}$ of diagonal matrices as in \eqref{diag-matrix} is called {\em dominated} if there exists a multicone $\mathcal{C} \subset \mathbb{P}^1$ such that $\bigcup_{u \in \mathcal U} A_u \mathcal{C} \subset \mathring{\mathcal{C}}$. It was shown in \cite[Theorem B]{BG} that $ (A_u)_{u \in \mathcal U}$ is dominated if and only if there exist constants $C>0$ and $0<\tau<1$ such that
\begin{equation}\label{q:dom}
    %\frac{\left|\operatorname{det} A_{\bm{u}}\right|}{\left\|A_{\bm{u}}\right\|^2} \leq C \tau^n
    \frac{b_{\bm{u}}\cdot c_{\bm{u}}}{\max\{b_{\bm{u}},c_{\bm{u}}\}^2} \leq C \tau^n
\end{equation}
for all $n \in \mathbb N$ and $\bm{u} \in \mathcal U^n$.

\medskip
For each diagonal matrix $A_u$ as in \eqref{diag-matrix}, the singular value function is given by
\[ \varphi^s (A_u) := \begin{cases}
\max \{b_u, c_u \}^s , & \text{if } 0 \le s < 1,\\
\max \{b_u, c_u\} \min \{b_u, c_u \}^{s-1}, & \text{if } 1 \le s < 2.
\end{cases}\]
The {\em Lyapunov exponents} of the collection $(A_u)_{u \in \mathcal U}$ with respect to a measure $\mu \in \mathcal{M}(\mathcal U^\mathbb N, \sigma)$ are defined as 
\[
\begin{aligned}
&\chi_1(\mu):= -\lim _{n\to\infty}\frac1n\int_{\mathcal U^\mathbb N}\log\max\left\{b_{\xi|_n},c_{\xi|_n}\right\}\,\mathrm{d}\mu(\xi),\\
&\chi_2(\mu):= -\lim _{n\to\infty}\frac1n\int_{\mathcal U^\mathbb N}\log\min\left\{b_{\xi|_n},c_{\xi|_n}\right\}\,\mathrm{d}\mu(\xi).
\end{aligned}\]
The Lyapunov dimension of $\mu \in \mathcal{M}(\mathcal U^\mathbb N, \sigma)$ is defined to be
\[ \operatorname{dim}_{\mathrm{L}}(\mu):= \min\left\{\frac{h_{\mu}(\sigma)}{\chi_1(\mu)}, 1+\frac{h_{\mu}(\sigma)-\chi_1(\mu)}{\chi_2(\mu)}\right\}.\]
For a continuous potential $\Phi: \mathcal U^\mathbb N \rightarrow \mathbb R^d$, $d \ge 1$, write $S_n \Phi=\sum_{k=0}^{n-1} \Phi \circ \sigma^k$ for its {\em Birkhoff sum}.  The {\em topological pressure} of $\Phi$ and $ (A_u)_{u \in \mathcal U}$ is given by
\[ P\left(\log \varphi^s+\Phi\right)=\lim _{n \rightarrow \infty} \frac{1}{n} \log \sum_{\bm{u} \in \mathcal U^n} \varphi^s(A_{\bm{u}}) \sup _{\xi \in[\bm{u}]} \exp \left(S_n \Phi(\xi)\right),\]
 where the existence of the limit is guaranteed by the sub-additivity of the potential.

\subsection{Hausdorff dimension}
For a subset $F \subseteq \mathbb R^n$, $n \ge 1$, and $\delta >0$, a $\delta$-cover of $F$ is a collection $\{U_i\}$ of subsets of $\mathbb R^n$ that each have diameter $\textrm{diam}(U_i)$ at most $\delta$ and satisfy $F \subseteq \bigcup_{i} U_i$. For $s >0$, the {\em $s$-dimensional Hausdorff outer measure} is defined as
\[ \mathcal H^s(F) = \lim_{d \downarrow 0} \inf \left\{ \sum_{i} \textrm{diam}(U_i)^s \, : \, \{ U_i\} \text{ is a $\delta$-cover of $F$} \right\}.\]
The {\em Hausdorff dimension} of the set $F$ is
\[ \operatorname{dim}_{\mathrm{H}}(F) = \inf\{s >0 \, : \, \mathcal H^s(F) =0\}.\]
Let $\mu$ be a finite Borel measure on $F$. The {\em Hausdorff dimension} of $\mu$ is
\[ \operatorname{dim}_{\mathrm{H}}(\mu) = \inf\{\operatorname{dim}_{\mathrm{H}}(Z)  \, : \, \mu(Z) =1\}.\]
The {\em lower pointwise dimension} of $\mu$ at a point $x \in F$ is defined by
\[ \underline{d}_{\mu}(x) = \liminf_{r \to 0} \frac{\log \mu(B(x,r))}{\log r},\]
where $B(x,r)$ denotes the open ball in $\mathbb R^n$ with radius $r$ centred at $x$. The following result can be found in e.g.~\cite[Theorem 7.1 and Theorem 7.2]{Pes97}.
\begin{lemma}\label{l:pwdim}
Let $F \subseteq \mathbb R^n$ be a Borel set and $\mu$ a finite Borel measure on $\mathbb R^n$. The following statements hold.
\begin{itemize}
\item[(i)] If $\underline{d}_{\mu}(x) \le c$ for some $c> 0$ and every $x \in F$, then $\operatorname{dim}_{\mathrm{H}}(F) \le c$.
\item[(ii)] If $\underline{d}_{\mu}(x) \ge c$ for some $c > 0$ and $\mu$-a.e.~$x \in F$, then $\operatorname{dim}_{\mathrm{H}}(\mu) \ge c$.
\end{itemize}
\end{lemma}

\section{Dominated diagonally affine IFSs}\label{s:ddaifs}
In this section, we prove Theorem~\ref{t:main1}. Recall the definition of the natural projection $\pi: \mathcal U^{\mathbb N} \to \Lambda $ from \eqref{q:pi}. Also recall the definitions of the sets $E_{\Phi}(\bm{\alpha})$ and $L_\Phi$ from \eqref{q:symblevel} and \eqref{q:Lphi}, respectively. We have the following upper bound for the Hausdorff dimension of $\pi(E_\Phi(\bm{\alpha}))$. In the proof, we make use of \cite[Proposition 3.2]{BJKR21}, which holds for general affine IFSs (including the diagonally affine case) on $\mathbb R^2$.

\begin{lemma}\label{l:u_upper}
Let $\left\{A_u+ \bm{v}_u\right\}_{u \in \mathcal U}$ be a diagonally IFS on $\mathbb R^2$ such that the collection $ (A_u)_{u \in \mathcal U}$  is dominated. Let $\Phi: \mathcal U^\mathbb N \rightarrow \mathbb{R}^d$, $d \ge 1$, be a continuous potential. Then for each $\bm{\alpha} \in \mathring{L}_{\Phi}$,
\[ \begin{aligned}
\operatorname{dim}_{\mathrm{H}}\left(\pi(E_{\Phi}(\bm{\alpha}))\right) & \le \sup \bigg\{\operatorname{dim}_{\mathrm{L}}(\mu)\, :\,  \mu \in \mathcal{M}(\mathcal U^\mathbb N, \sigma) \text { and } \int_{\mathcal U^\mathbb N} \Phi\, \mathrm{d} \mu=\bm{\alpha} \bigg\} \\
& =\sup \bigg\{s \geq 0\, :\,  \inf _{q \in \mathbb{R}^d} P\left(\log \varphi^s+\langle q, \Phi-\bm{\alpha}\rangle\right) \geq 0\bigg\}.
\end{aligned}\]
\end{lemma}

\begin{proof}
It follows directly from \cite[Lemma 3.1 and Proposition 3.2]{BJKR21} that for any diagonally affine IFS $\{A_u+ \bm{v}_u\}_{u\in \mathcal U}$ on $\mathbb R^2$ and continuous potential $\Phi: \mathcal U^\mathbb N \to \mathbb R^d$, $d\ge 1$, and any $\bm{\alpha} \in \mathring{L}_{\Phi}$,
\begin{equation}\label{q:upperbound}
\dim_{\mathrm{H}} (\pi(E_{\Phi}(\bm{\alpha}))) \leq  \sup\left\{ s \ge 0 \, : \, \inf_{q \in \mathbb R^{d}} P (\log \varphi^s + \langle q,\Phi-\bm{\alpha} \rangle) \ge 0 \right\}.
\end{equation}
A measure $\nu \in \mathcal{M}(\mathcal U^\mathbb N, \sigma^{n})$ is called an {\em $n$-step Bernoulli measure} if it is a Bernoulli measure on $\left(\mathcal U^\mathbb N, \sigma^{n} \right)$. For  $n$-step Bernoulli measures $\nu \in \mathcal M (\mathcal U^\mathbb N, \sigma^{n})$, write
\begin{equation}\label{q:tildemu}
\tilde \nu = \frac1n \sum_{k=0}^{n-1} \nu \circ \sigma^{-k}.
\end{equation}
Then $\tilde \nu \in \mathcal M (\mathcal U^\mathbb N, \sigma)$ and $\tilde \nu$ is ergodic. Since $( A_u)_{u \in \mathcal U}$ is dominated, it follows by \cite[Proposition 4.3]{BJKR21} and \eqref{q:upperbound} that for any $\bm{\alpha} \in \mathring{L}_{\Phi}$,
\[
\begin{split}
 \sup & \left\{s \geq 0: \inf _{q \in \mathbb{R}^d} P\left(\log \varphi^s+\langle q, \Phi-\bm{\alpha}\rangle\right) \geq 0\right\} \\
\le \ &  \sup \bigg\{\dim_{\mathrm{L}}(\tilde{\nu}): \nu  \text { fully supported } n\text{-step  Bernoulli and } \int_{\mathcal U^\mathbb N} \Phi \, \mathrm{d} \tilde{\nu}=\bm{\alpha}  \bigg\} \\
\le \ &  \sup \left\{\dim_{\mathrm{L}}(\mu): \mu \in \mathcal{M}(\mathcal U^\mathbb N, \sigma)\text { and } \int_{\mathcal U^\mathbb N} \Phi \, \mathrm{d} \mu=\bm{\alpha}\right\}.
\end{split}
\]
On the other hand, if we let $\mu \in \mathcal{M}(\mathcal U^\mathbb N, \sigma)$ be such that $\int_{\mathcal U^\mathbb N} \Phi\,  \mathrm{d} \mu=\bm{\alpha}$, then for any $0 \le t<\operatorname{dim}_{\mathrm{L}}(\mu)$, it holds by the sub-additive variational principle (see \cite{CFH08}) that for all $q \in \mathbb{R}^d$,
\[ P\left(\log \varphi^t+\langle q, \Phi-\bm{\alpha}\rangle\right) \geq h_{\mu}(\sigma)+\lim _{n \rightarrow \infty} \frac{1}{n} \int_{\mathcal U^\mathbb N} \log \varphi^t\left(A_{\left.\xi\right|_n}\right)\,  \mathrm{d} \mu(\xi)>0.\]
Hence, $t \leq \sup \left\{s \geq 0: \inf _{q \in \mathbb{R}^d} P\left(\log \varphi^s+\langle q, \Phi-\bm{\alpha}\rangle\right) \geq\right.$ $0\}$ and thus,
\[ \begin{split}
\sup \bigg\{\dim_{\mathrm{L}}(\mu) \,:\,   \mu \in \mathcal{M}(\mathcal U^\mathbb N, \sigma) & \text{ and } \int_{\mathcal U^\mathbb N} \Phi \, \mathrm{d} \mu=\bm{\alpha}\bigg\} \\
 \le \ & \sup \left\{s \geq 0: \inf _{q \in \mathbb{R}^d} P\left(\log \varphi^s+\langle q, \Phi-\bm{\alpha}\rangle\right) \geq 0\right\} .
\end{split}\]
This gives the result.
\end{proof}

\begin{remark}\label{eq: Lemma}
{\rm Note that the proof of Lemma~\ref{l:u_upper} shows that in fact
\[ \begin{split}
 \sup & \left\{s \geq 0: \inf _{q \in \mathbb{R}^d} P\left(\log \varphi^s+\langle q, \Phi-\bm{\alpha}\rangle\right) \geq 0\right\} \\
= \ &  \sup \bigg\{\dim_{\mathrm{L}}(\tilde{\nu}): \nu  \text { fully supported } n\text{-step  Bernoulli and } \int_{\mathcal U^\mathbb N} \Phi \, \mathrm{d} \tilde{\nu}=\bm{\alpha}  \bigg\} \\
= \ &  \sup \left\{\dim_{\mathrm{L}}(\mu): \mu \in \mathcal{M}(\mathcal U^\mathbb N, \sigma)\text { and } \int_{\mathcal U^\mathbb N} \Phi \, \mathrm{d} \mu=\bm{\alpha}\right\}.
\end{split}\]
}\end{remark}

Under the additional conditions mentioned in the statement of Theorem~\ref{t:main1}, we can prove that this upper bound in fact equals the Hausdorff dimension of the level set. Note that it would be possible to combine Hochman \cite{Hochman14} and Jordan and Simon \cite{JS} to obtain a similar result for almost all vectors $\bm{v}_u$ but our Theorem \ref{t:main1} is proved for all vectors $\bm{v}_u$. \cite{BJKR21} proved a similar result for affine IFSs satisfying the SOSC under the assumption that the set of matrices $ (A_u)_{u \in \mathcal U}$ is strongly irreducible such that the generated subgroup of the normalised matrices is not relatively compact. Theorem \ref{t:main1} is inspired by their result.

%Let $P_x, P_y$ denote projection onto the corresponding coordinate axis.

\begin{proof}[Proof of Theorem~\ref{t:main1}]
Let $\{ A_u + \bm{v}_u\}_{u \in \mathcal U} \in \mathcal D$. For each $u \in \mathcal U$, it holds that
\[ \frac{b_u \cdot c_u}{ \max\{ b_u, c_u\}^2} = \frac{\min \{ b_u, c_u\} }{ \max\{ b_u, c_u\}} <1,\]
since either $|b_u| > |c_u|$ for all $u \in \mathcal U$ or $|b_u| < |c_u|$ for all $u \in \mathcal U$. Take $\tau = \max_{u \in \mathcal U} \left\{ \frac{\min \{ b_u, c_u\} }{ \max\{ b_u, c_u \}} \right\}$. Then $\tau < 1$ and, since each $A_u$ is a diagonal matrix, we get \eqref{q:dom} with $C=1$. Hence, $\{A_u\}_{u \in \mathcal U}$ is dominated and therefore the desired upper bound for the Hausdorff dimension of $\pi(E_{\Phi}(\bm{\alpha}))$ is given by Lemma~\ref{l:u_upper}. 

\medskip
For the lower bound, suppose that $\nu \in \mathcal M(\mathcal U^\mathbb N, \sigma^{n})$ is a fully supported $n$-step Bernoulli measure with $\int_{\mathcal U^\mathbb N} \Phi\, \mathrm{d} \tilde{\nu}=\bm{\alpha}$ with $\tilde \nu$ as defined in \eqref{q:tildemu}. The existence of the measure $\nu$ is guaranteed by \cite[Proposition 4.3]{BJKR21}. Then $\tilde \nu \in \mathcal M (\mathcal U^\mathbb N, \sigma)$ and $\tilde \nu$ is ergodic and therefore from $\int_{\mathcal U^\mathbb N} \Phi\, \mathrm{d} \tilde{\nu}=\bm{\alpha}$, we get that
\begin{equation}\label{full-mea}
\tilde{\nu}\left(\left\{\xi \in \mathcal U^\mathbb N: \lim _{n \rightarrow \infty} \frac{1}{n} S_n \Phi (\xi)=\bm{\alpha}\right\} \right)=1. 
\end{equation}

\medskip
Let $\hat{\nu}=\tilde{\nu} \circ \pi^{-1} $. Assume that $|b_u|> |c_u|$ for all $u \in \mathcal U$ so that we are in the situation of condition (D) (the proof for the case (D') goes similarly). Then the strong stable direction of the collection $(A_u)_{u \in \mathcal U}$ is equal to the subspace parallel to the $y$-axis (see \cite{BG}). Let $P_x \hat{\nu}$ be the measure on $[0,1]$ given by the canonical projection onto the $x$-coordinate of $\hat \nu$. Since the matrices $A_u$ are diagonal, $P_x \hat{\nu}$ is a self-similar measure for the IFS $\mathcal G_1$, i.e.\! there is a probability vector $\hat{\bm{p}} = (\hat p_u)_{u \in \mathcal U}$ such that
\[P_x \hat{\nu} (B)=\sum_{u \in \mathcal U} \hat p_u P_x \hat{\nu} (g_{1,u}^{-1}(B))\] 
for each Borel set $B \subseteq [0,1]$. Then condition (D)(a) together with \cite[Theorem 1.1]{Hochman14} or condition (D)(b) together with \cite[Theorem 1.2]{rapaport} yields
\begin{equation}\label{q:nodrop}
\dim_{\mathrm{H}} (P_x \hat{\nu})=\min\bigg\{1, \frac{h_{\tilde{\nu}}(\sigma)}{\chi_1(\tilde{\nu})}\bigg\}.
\end{equation}
 It then follows from \cite[Corollaries 2.7 and 2.8]{BARANY2017} and \eqref{q:nodrop} that
\[ \dim_{\mathrm{H}} (\hat{\nu})=\operatorname{dim}_{\mathrm{L}}(\tilde{\nu}).\]
This and \eqref{full-mea} yield $\dim_{\mathrm{H}}\left(\pi(E_{\Phi}(\bm{\alpha}))\right) \geq \operatorname{dim}_{\mathrm{L}}(\tilde{\nu})$. Since this holds for arbitrary fully supported $n$-step Bernoulli measures $\nu$ with $\int_{\mathcal U^\mathbb N} \Phi\, \mathrm{d} \tilde{\nu}=\bm{\alpha}$, the result follows from Remark \ref{eq: Lemma}. 
\end{proof}

\begin{remark}{\rm We make a small remark on the conditions (D) and (D'). It was shown by Hochman in \cite[proof of Theorem 1.5]{Hochman14} that an IFS satisfies the ESC if it does not have exact overlaps and all parameters $b_u, c_u,\beta_u, \gamma_u$ are algebraic numbers over $\mathbb Q$. In \cite{Chen_21, Baker-ESC, BK-ESC} it was shown that there exist iterated function systems that do not contain exact overlaps while there are cylinders which are super-exponentially close at all small scales, i.e.\! the ESC does not hold. What is needed in the proof of Theorem~\ref{t:main1} is \eqref{q:nodrop}, which is also guaranteed by \cite{rapaport} under the assumption of having algebraic $b_u,c_u$ and no exact overlaps. 
}\end{remark}

\section{Digit frequencies for finite GLS expansions}\label{s:fgls}
We now move to the second type of IFS we consider. Fix a GLS IFS $\{A_e +\bm{v}_e\}_{e \in \mathcal E}$. We start by proving some properties of the expansions from \eqref{q:fglse}.

\subsection{Multiple representations}
First, consider the representations of the points $y \in [0,1]$. Recall the definition of the maps $f_j$ from \eqref{q:fj}. The IFS $\{f_j \}_{0 \le j < J}$ satisfies the SOSC and has the interval $[0,1]$ as its attractor. Let $\pi_J: \{0,1, \ldots, J-1\}^\mathbb N \to [0,1]$ be the map given by
\[ \pi_J((j_m)_{m \ge 1}) = \lim_{m \to \infty}f_{j_1} \circ f_{j_2} \circ \cdots \circ f_{j_m}(0).\]
One easily sees that to all but countably many $y \in [0,1]$, there corresponds a unique sequence $\zeta \in \{0,1, \ldots, J-1\}^\mathbb N$ such that $y = \pi_J(\zeta)$ and otherwise $\#\pi^{-1}_J\{y\} =2$ and there is one sequence ending in an infinite string of 0's and one ending in an infinite string of $(J-1)$'s. We make the following observation, which we will use later on. Recall the definition of the set $W(\bm{\alpha})$ from \eqref{q:walphadef}.

\begin{lemma}\label{l:uniquew}
    Let $\bm{\alpha} = (\alpha_e)_{e \in \mathcal E} \in [0,1]^{\# \mathcal E}$ be a frequency vector with $\alpha_j>0$ for each $0 \le j < J$.
    Then $\#\pi_J^{-1}\{y\}=1$ for any $y \in W(\bm{\alpha})$.
    % {\color{cyan}In addition, for each $y\in W({\bm\alpha})$, $\#\pi^{-1}\{(x,y)\}>1$ for only countably many $x\in[0,1]$.}
\end{lemma}

\begin{proof}
    Let $y \in [0,1]$ be such that $\#\pi_J^{-1}\{y\}=2$ and let $x \in [0,1]$. Then any $\omega = (j_m,k_m)_{m \ge 1} \in \pi^{-1}\{(x,y)\}$ either has $j_m=0$ for all $m$ large enough or $j_m = J-1$ for all $m$ large enough. In the first case, $\sum_{k=0}^{B_j-1}\tau_{(j,k)}(\omega) =0 \neq \alpha_j$ for all $j \neq 0$ and in the second case, $\sum_{k=0}^{B_j-1} \tau_{(j,k)}(\omega) =0 \neq \alpha_j$ for all $j \neq J-1$. Hence, $(x,y) \not \in F(\bm{\alpha})$ and thus $W(\bm{\alpha}) = \emptyset$.
\end{proof}

For a fixed $y \in [0,1]$, we can consider the expansions one obtains from the GLS IFS for $x \in [0,1]$. We define the {\em fibre fundamental intervals} corresponding to $y$ by setting for each $m \ge 1$ and $k_1, \ldots, k_m$ satisfying $0 \le k_i < B_{j_i}$ for all $1 \le i \le m$,
\begin{equation}\label{q:fibcyls}
\Delta_y(k_1, \ldots, k_m) := h_{(j_1,k_1)} \circ \cdots \circ h_{(j_m,k_m)} ([0,1]),
\end{equation}
where we let $(j_m)_{m \ge 1}$ be the lexicographically smallest sequence in $ \pi_J^{-1}\{y\}$. For $y \in W(\bm{\alpha})$ this means that $(j_m)_{m \ge 1} \in \pi_J^{-1}\{y\}$ is the unique sequence that satisfies $\tau_j((j_m)_{m \ge 1})=\alpha_j$ for each $0 \le j < J$ and for $y \in [0,1]\setminus W(\bm{\alpha})$ the sequence $(j_m)_{m \ge 1}$ is the one ending in an infinite string of $(J-1)$'s.
\medskip

If we fix $y\in W({\bm \alpha})$ and take $x\in[0,1]$ such that $\#\pi^{-1}\{(x,y)\}>1$, then we know by Lemma~\ref{l:uniquew} that $\#\pi_J^{-1}\{y\}=1$, say $y=\pi_J((j_m)_{m\geq1})$.
Consequently, $x$ must have multiple expansions along the fibre $y$ and so must lie on the boundary of a fibre fundamental interval $\Delta_y(k_1,\dotsc,k_m)$ for some $0\le k_i<B_{j_i}$ ($1\le i\le m$) and some $m\in\mathbb{N}$.
Since each fibre fundamental interval has two boundary points and there are only countably many fibre fundamental intervals, the set of such points $x$ must be countable.
\medskip

Now, fix an $x \in [0,1]$. Since the GLS IFS $\{A_e + \bm{v}_e\}_{e \in \mathcal E}$ has $[0,1]^2$ as its attractor, to any $y \in [0,1]$, there corresponds a sequence $\omega \in \mathcal E^\mathbb N$ such that $\pi(\omega)=(x,y)$. Therefore, to any sequence $(j_m)_{m \ge 1} \in \{0,1, \ldots, J-1\}^\mathbb N$, there corresponds a sequence $(k_m)_{m \ge 1}$ with $0 \le k_m < B_{j_m}-1$ for each $m \in \mathbb N$ such that
\[ x= \lim_{m \to \infty} h_{(j_1,k_1)} \circ \cdots \circ h_{(j_m,k_m)}(0).\]
We show that if $h_e \neq h_e'$ whenever $e \neq e'$, then each of the sequences $(j_m)_{m \ge 1} \in \{0,1, \ldots, J-1\}^\mathbb N$ yields a different GLS expansion for $x$ as in \eqref{q:fglse}.

\medskip
The GLS expansions from \eqref{q:fglse} are given by the triples of digits $(s_m,K_m,t_m)$, $m \in \mathbb N$, from \eqref{q:smKmtm}. Therefore, if we set
\[ \mathcal A =  \left\{ \left(\varepsilon_e, q_e^{-1}, r_e+ \varepsilon_e q_e \right) \, : \, e \in \mathcal E \right\},\]
then we can think of $\mathcal A$ as the {\em GLS digit set} corresponding to $\{ A_e + \bm{v}_e\}_{e \in \mathcal E}$ and we can map sequences $ ((j_m,k_m))_{m \ge 1} \in \mathcal E^\mathbb N$ to sequences $(s_m,K_m,t_m)_{m \ge 1} \in \mathcal A^\mathbb N$ through the identification given in \eqref{q:smKmtm}. Let $(j_m)_{m \ge 1}, (j_m')_{m \ge 1} \in \{0,1, \ldots, J-1\}^\mathbb N$ be two different sequences, so there is an $m \in \mathbb N$ such that $j_m \neq j_m'$. Let $(j_m,k_m)_{m \ge 1}, (j_m',k_m')_{m \ge 1} \in \mathcal E^\mathbb N$ be two sequences that both project to $x$ in the second coordinate under $\pi$. Since $j_m \neq j_m'$, it holds that $(j_m,k_m) \neq (j_m',k_m')$. If we assume that $h_e \neq h_e'$ whenever $e \neq e'$, then it would follow that $\varepsilon_{(j_m,k_m)} \neq \varepsilon_{(j_m',k_m')}$ or $q_{(j_m,k_m)} \neq q_{(j_m',k_m')}$ and thus that the digits from $\mathcal A$ corresponding to $(j_m,k_m)$ and $(j_m',k_m')$ differ. Therefore, we immediately find the following result.

\begin{prop}\label{p:uncountable}
Let $\{A_e +\bm{v}_e\}_{e \in \mathcal E}$ be a GLS IFS with the additional assumption that $h_e \neq h_{e'}$ whenever $e \neq e'$.
Then for each $x \in [0,1]$, there are uncountably many different digit sequences $(s_m,K_m,t_m)_{m \ge 1} \in \mathcal A^\mathbb N$ with 
\[ x =\sum_{m\ge 1} (-1)^{\sum_{i=1}^{m-1}s_i} \, \frac{t_m}{\prod_{i=1}^m{K_i}}.\]
\end{prop}

The above also shows that there is a one-to-one correspondence between the sequences in $\mathcal E^\mathbb N$ and in $ \mathcal A^\mathbb N$, which justifies considering the elements of $\mathcal E$ as digits in the GLS expansions.

\subsection{Non-empty level sets} 
In this section, we determine for which frequency vectors $\bm{\alpha}$ the level set $F(\bm{\alpha})$ from \eqref{q:Falpha} and the set $W(\bm{\alpha})$ from \eqref{q:walphadef} are non-empty. We first consider the level sets $F(\bm{\alpha})$. 

\begin{prop}\label{SpectrumFull2D}
The set $F(\bm{\alpha})$ is non-empty for any frequency vector $\bm{\alpha} = (\alpha_e)_{e \in \mathcal E} \in [0,1]^{\# \mathcal E}$. 
\end{prop}

\begin{proof}
Fix a frequency vector $\bm{\alpha} =(\alpha_e)_{e \in \mathcal E} \in [0,1]^{\# \mathcal E}$. It is sufficient to construct a sequence $\omega = (\omega_n)_{n \ge 1} \in \mathcal E^\mathbb N$ such that $\tau_e(\omega)=\alpha_e$ for each $e \in \mathcal E$ since then $\pi(\omega) \in F(\bm{\alpha})$. Denote by $\lfloor\cdot\rceil$ the nearest-integer function. Order the elements in $\mathcal E$ by setting $(j,k) \prec (j',k')$ if either $j < j'$ or if $j=j'$ and $k<k'$. For each $n \ge 1$, set
\[ E_n = \left\{ e \in \mathcal E \, : \, \lfloor n \alpha_e \rceil = \lfloor (n-1) \alpha_e \rceil +1 \right\} = \left\{ e_{n,1} \prec \cdots \prec e_{n,m_n} \right\},\]
where $E_n$ can be empty and thus $m_n=0$ for some $n$. Define $\omega \in \mathcal E^\mathbb N$ by setting for each $n \ge 1$ and $1 \le m \le m_n$,
\[ \omega_{m+ \sum_{i=1}^{n-1} m_i} = e_{n,m},\]
where we let $\sum_{i=1}^0 m_i=0$. Clearly, there are infinitely many $n$ for which $E_n \neq \emptyset$ so $\omega$ is well defined.

\medskip
Now observe that for each $n$ the number of terms of $\omega$ we have defined using $\bigcup_{i=1}^n E_i$ is
\[ \sum_{i=1}^n m_i = \sum_{e\in \mathcal E}\lfloor n\alpha_e \rceil.\]
Since $\sum_{e\in \mathcal E}n\alpha_e=n$ and $\mathfrak m:=\# \mathcal E<\infty$, we must have
\[ n- \mathfrak m \le \sum_{i=1}^n m_i \le n+\mathfrak m.\]
Thus for each $e \in \mathcal E$,
\[ \begin{split}
\frac{\#\{0\le m <n\, :\, \omega_m=e\}}{n} \le \ & \frac{\#\{0\le m < \mathfrak m + \sum_{i=1}^n m_i \, :\, \omega_m=e\}}{n}\\
\le \ & \frac{\#\{0\le m < \sum_{i=1}^n m_i \, : \, \omega_m=e\} + \mathfrak m }{n}\\
\le \ & \frac{n \alpha_e +1 + \mathfrak m }{n}.
\end{split}\]
Similarly, it holds that
\[ \frac{\#\{0\le m <n:\omega_m=e\}}{n} \ge  \frac{n \alpha_e -1-\mathfrak m}{n}.\]
Taking the limit as $n \to \infty$ yields $\tau_e(\omega) = \alpha_e$ for all $e \in \mathcal E$.
\end{proof}

For a fixed frequency vector $\bm{\alpha}$, we would like to determine the set $W(\bm{\alpha})$ of points $y\in [0,1]$ for which there exists an $x\in [0,1]$ such that the point $(x,y)$ has digit frequencies given by $\bm{\alpha}$, see \eqref{q:walphadef}. Recall the definition of the Borel measure $\nu_{\bm\alpha}=\mu_{\bm\alpha}\circ\pi_2^{-1}$ from the introduction. We have the following result.

\begin{prop}\label{Walpha}
Let $\bm{\alpha} = (\alpha_e)_{e \in \mathcal E} \in [0,1]^{\# \mathcal E}$. Then
$$W({\bm\alpha})=\left\{y\in[0,1]\, :\, \exists \, \omega \in  \pi_2^{-1} \{ y\} \, \text{ s.t. } \,  \sum_{0 \le k < B_j} \tau_{(j,k)}(\omega)= \alpha_j \,\text{ for all } 0 \le j < J\right\}.$$
In particular, $\nu_{\bm\alpha}(W({\bm\alpha}))=1$.
\end{prop}

\begin{proof}$ $\medskip

\noindent $(\subset)$ Set
\[ W = \left\{y\in[0,1]\, :\, \exists \, \omega \in  \pi_2^{-1} \{ y\} \, \text{ s.t. } \,  \sum_{0 \le k < B_j} \tau_{(j,k)}(\omega)= \alpha_j \,\text{ for all } 0 \le j < J\right\}.\]
First, let $y \in W(\bm{\alpha})$. This means that $F_y({\bm\alpha})\neq\emptyset$ so there exists an $\omega\in\mathcal E^\mathbb N$ such that $\pi(\omega)=(x,y)$ for some $x\in[0,1]$ and $\tau_e(\omega)=\alpha_e$ for all $e\in\mathcal E$. Consequently, we find that
$$\sum_{0\leq k<B_j}\tau_{(j,k)}(\omega)=\sum_{0\leq k<B_j}\alpha_{(j,k)} = \alpha_j$$ for all $0 \le j < J$ and so $y \in W$.

\medskip
\noindent $(\supset)$ Conversely, let $y\in W$.
Then there exists an $\omega' = ( j_\ell', k_\ell')_{\ell \ge 1} \in  \pi_1^{-1} \{ y\}$ for which it holds that $\sum_{0 \le k < B_j} \tau_{(j,k)}(\omega')= \alpha_j$ for all $0 \le  j < J$.
% To complete the proof, it suffices to construct a sequence $\omega\in\mathcal E^\mathbb N$ such that $\tau_e(\omega)=\alpha_e$ for all $e\in\mathcal E$ and $\pi(\omega)=(w,x)$ for some $x\in[0,1]$ since then $x\in F_w(\alpha)$. So,
Write $\zeta=(j_\ell')_{\ell \ge 1} \in \{0, \ldots, J-1\}^\mathbb N$. For each $0 \le j < J$ and $n \ge 1$, set
\[ E_n^{(j)}:= \left\{ (j,k)\in\mathcal E \, :\, \left\lfloor \frac{n\alpha_{(j,k)}}{\alpha_j}\right\rceil=\left\lfloor\frac{(n-1)\alpha_{(j,k)}}{\alpha_j}\right\rceil+1\right\}=\left\{e_{n,1}^{(j)}\prec\dotsb\prec e_{n,m_n^{(j)}}^{(j)}\right\}\] 
and let $\omega^{(j)} \in \mathcal E^\mathbb N$ be the sequence obtained from concatenating all elements from the sets $E_n^{(j)}$ as in Proposition~\ref{SpectrumFull2D}, so
\[ \omega^{(j)} = e_{1,1}^{(j)} \cdots e_{1,m_1^{(j)}}^{(j)} e_{2,1}^{(j)} \cdots e_{2,m_2^{(j)}}^{(j)} e_{3,1}^{(j)} \cdots.\]
% that for each $n \ge 1$ and $1 \le m \le m_n^{(j)}$,
%\[ \omega^{(j)}_{m+ \sum_{i=1}^{n-1} m_i^{(j)}} = e_{n,m}^{(j)}.\]
Then as in Proposition~\ref{SpectrumFull2D}, we obtain that for each $0 \le k < B_j$,
\[ \lim_{n \to \infty} \frac{\# \{ 1 \le m \le n \, : \, \omega^{(j)}_m = (j,k) \} }{n} = \frac{\alpha_{(j,k)}}{\alpha_j}.\]
We now weave the sequences $\omega^{(j)}$ together to construct a sequence $\omega = (j_\ell', k_\ell)_{\ell \ge 1} \in \mathcal E^\mathbb N$ that satisfies $\tau_e(\omega) = \alpha_e$ for all $e \in \mathcal E$. Then $\pi(\omega) = (x,y)$ for some $x \in F_y(\bm{\alpha})$, which shows that $F_y(\bm{\alpha}) \neq \emptyset$. For each $\ell \ge 1$, let $\omega_\ell$ be the $\tau_{j_\ell'}(\zeta,\ell)$-th element of the sequence $\omega^{(j_\ell')}$. So, $\omega_1 = \omega^{(j'_1)}_1  = e_{1,1}^{(j'_1)}$, $\omega_2$ either equals $\omega^{(j_1')}_2$ if $j'_1=j'_2$ or $\omega^{(j_2')}_1$ if $j'_1 \neq j'_2$, et cetera. As the sequences in the first coordinates of $\omega = (j_\ell', k_\ell)_{\ell \ge 1}$ and $\omega' = (j_\ell', k_\ell')_{\ell \ge 1}$ coincide, we have for each $0 \le j < J$,
\[ \sum_{0 \le k < B_j} \tau_{(j,k)}(\omega) = \sum_{0 \le k < B_j} \tau_{(j,k)}(\omega')= \alpha_j. \]
Moreover, for each $e = (j,k) \in \mathcal E$ and $n \ge 1$,
\begin{equation}\label{q:tildetau}
 \# \{ 1 \le m \le n \, : \, \omega_m =e \} = \# \{ 1 \le m \le \tau_j(\zeta,n) \, : \, \omega^{(j)}_m = e\}.
\end{equation}
If $\alpha_j>0$, then $\tau_j(\zeta,n) >0$ for all $n$ large enough and for any $e = (j,k) \in \mathcal E$, we obtain
\[  \tau_e(\omega) = \lim_{n \to \infty} \frac{\# \{ 1 \le m \le \tau_j(\zeta,n) \, : \, \omega^{(j)}_m = e\} }{\tau_j(\zeta,n)} \cdot \frac{\tau_j(\zeta,n)}{n} =  \frac{\alpha_{(j,k)}}{\alpha_j}  \sum_{0 \le k < B_j} \tau_{(j,k)}(\omega) = \alpha_{(j,k)}.\]
If $ \alpha_j=0$, then $\alpha_{(j,k)}=0$ for each $0 \le k < B_j$ and by \eqref{q:tildetau},
\[ 0 \le \tau_{(j,k)}(\omega) \le \lim_{n \to \infty} \frac{\tau_j(\zeta,n)}{n} = \alpha_j =0.\]
This gives the first part of the statement. 

\medskip
As $\left\{ \omega \in \mathcal E^\mathbb N \, : \, \tau_e(\omega) = \alpha_e \text{ for all } e \in \mathcal E \right\} \subseteq \pi_1^{-1}(W(\bm{\alpha}))$, it follows from the definition of $\mu_{\bm{\alpha}}$ that $\nu_{\bm\alpha}(W({\bm\alpha}))=1$.
\end{proof}

\subsection{The Hausdorff dimension of the Besicovitch-Eggleston Sets}
In this section, we prove Theorem~\ref{t:main4}. 
The proof is similar to \cite[Theorem~3.1]{BI09} and \cite[Theorem~1.1]{FLMW10}, which both treat digit frequencies for expansions with infinite digit sets that can be generated by an IFS on $\mathbb R$ as in \eqref{q:singlegls}. Their results do not apply to our setting because the IFS $\{ h_e:[0,1] \to [0,1]\}_{e \in \mathcal E}$ on $\mathbb R$ is not of this type. Nevertheless, since we have a finite digit set, we can adapt the method of proof from \cite[Theorem~3.1]{BI09}.\medskip

\medskip
Fix a $y \in [0,1]$. Recall the definition of the fibre fundamental intervals $\Delta_y (k_1, \ldots, k_m)$ from \eqref{q:fibcyls}. Note that we obtain a semi-algebra of sets generating the Borel $\sigma$-algebra $\mathcal B([0,1])$ on $[0,1]$ by taking the collection of all intervals (open, closed and half-open) that can be formed by the endpoints of the fibre fundamental intervals. Suppose that the frequency vector $\bm{\alpha}$ satisfies the following additional property: for each $0 \le j < J$ there are $k, \ell \in B_j$ with $k \neq \ell$ and $\alpha_{(j,k)}>0$ and $\alpha_{(j, \ell)}>0$. Let $m_{y,{\bm\alpha}}$ be the measure on $([0,1], \mathcal B([0,1]))$ determined by
$$m_{y,{\bm\alpha}}(\Delta_y(k_1, \ldots, k_m))=\prod_{i=1}^m \frac{\alpha_{(j_i,k_i)}}{\alpha_{j_i}}, \quad 0 \le k_i < B_{j_i}, \, 1 \le i \le m, \, m \ge 1,$$
and by the same quantity for any interval determined by the same endpoints. This immediately implies that any endpoint $x$ of a fibre fundamental interval has $m_{y,{\bm\alpha}}(\{x\})=0$. If $x\in [0,1]$ is not an endpoint of a fibre fundamental interval, then there is a sequence $(k_m)_{m \ge 1}$ such that
\[ \bigcap_{m \in \mathbb N} \Delta_y(k_1, \ldots, k_m) = \{x\}.\]
This implies that
\[m_{y,{\bm\alpha}}(\{x\}) = \lim_{m \to \infty} \prod_{i=1}^m \frac{\alpha_{(j_i,k_i)}}{\alpha_{j_i}}.\]
By the additional assumption on $\bm{\alpha}$ there is a constant $0< c <1$ such that $\frac{\alpha_{j,k}}{\alpha_j}<c <1$ for all $(j,k) \in \mathcal E$. Therefore, $\mu_{y,\bm{\alpha}}(\{x\})=0.$  We will need the following property of $m_{y,{\bm\alpha}}$.

\begin{lemma}\label{Full}
For $\nu_{\bm\alpha}$-a.e.~$y\in W({\bm\alpha})$, it holds that $m_{y,{\bm\alpha}}(F_y({\bm\alpha}))=1$.
\end{lemma}

\begin{proof}
	Observe that for each $(j,k)\in\mathcal E$ by Proposition~\ref{Walpha},
	\begin{equation} \begin{split}\label{CondMeas}
		\int_{[0,1]}\int_{[0,1]} & \mathbbm 1_{\pi([(j,k)])}(x,y)\,\mathrm{d}m_{y,{\bm\alpha}}(x)\,\mathrm{d}\nu_{\bm\alpha}(y)\\
            &=\int_{[\sum_{i=0}^{j-1}p_i, \sum_{i=0}^j p_i]}
            \int_{[0,1]}\mathbbm 1_{\Delta_y(k)}(x)\,\mathrm{d}m_{y,{\bm\alpha}}(x)\,\mathrm{d}\nu_{\bm\alpha}(y)\\
            &= \frac{\alpha_{(j,k)}}{\alpha_j} 
            \nu_{\bm\alpha}\left(\left[\sum_{i=0}^{j-1}p_i, \sum_{i=0}^j p_i\right] \right)\\
	    &=\alpha_{(j,k)} =\int_{[0,1]^2}\mathbbm 1_{\pi([(j,k)])}\,\mathrm{d}\mu_{\bm\alpha}\circ\pi^{-1}.
	\end{split}\end{equation}
	Since the collection $\{\pi([e_1, \ldots , e_n]): e_i\in\mathcal E, \,\, 1 \le i \le n \}$ generates the Borel $\sigma$-algebra $\mathcal B([0,1]^2)$, we can conclude from \eqref{CondMeas} that
 \begin{equation}\label{q:equalf}
 \int_{[0,1]} \int_{[0,1]} f \, \mathrm{d}m_{y,{\bm\alpha}}\,\mathrm{d}\nu_{\bm\alpha} = \int_{[0,1]^2} f \, \mathrm{d}\mu_{\bm\alpha}\circ\pi^{-1} 
 \end{equation}
for all $f\in L^1([0,1]^2,\mathcal B([0,1]^2),\mu_{\bm\alpha}\circ\pi^{-1})$.

\medskip
Let $E := \{ y \in W(\bm{\alpha}) \, : \, m_{y, \bm{\alpha}} (F_y(\bm{\alpha})) < 1 \}$ and suppose that $\nu_{\bm{\alpha}} (E) >0$. From \eqref{q:equalf} with $f=\mathbbm 1_{F({\bm\alpha})}$ together with Proposition~\ref{Walpha}, we then find that
	\begin{align*}
		\mu_{\bm\alpha}\circ\pi^{-1}(F({\bm\alpha}))
						&=\int_{W({\bm\alpha})}\int_{[0,1]}\mathbbm 1_{F_y({\bm\alpha})}(x)\,\mathrm{d}m_{y,{\bm\alpha}}(x)\,\mathrm{d}\nu_{\bm\alpha}(y)\\
						&=\int_{W({\bm\alpha})\setminus E}1\,\mathrm{d}\nu_{\bm\alpha}
                            +\int_E\int_{[0,1]}\mathbbm 1_{F_y({\bm\alpha})}(x)\,\mathrm{d}m_{y,{\bm\alpha}}(x)\,\mathrm{d}\nu_{\bm\alpha}(y)\\
                            &<\int_{W({\bm\alpha})}1\,\mathrm{d}\nu_{\bm\alpha}\\
						&=1.
	\end{align*}
On the other hand, recall that $E_{\mathbbm 1}(\bm{\alpha})$ is the symbolic Besicovitch-Eggleston set containing all sequences $\omega \in \mathcal E^\mathbb N$ with $\tau_e (\omega)=\alpha_e$ for each $e \in \mathcal E$. Therefore, by the definition of $\mu_{\bm{\alpha}}$,
\[ 1 = \mu_{\bm{\alpha}}(E_{\mathbbm 1}(\bm{\alpha})) \le \mu_{\bm{\alpha}} \circ \pi^{-1} (F(\bm{\alpha})) < 1.\]
This gives a contradiction. It follows that $m_{y,{\bm\alpha}}(F_y({\bm\alpha}))=1$ for $\nu_{\bm\alpha}$-a.e.~$y\in W({\bm\alpha})$.
\end{proof}

%\medskip
%For $(j,k) \in \mathcal E$ put $q_{(j,k)} = r_{(j,k+1)}-r_{(j,k)}$.

Before we move to the proof of Theorem~\ref{t:main4}, to simplify notation we put $p_{(j,k)}=p_j$ for all $(j,k)\in\mathcal E$. Also, let $\mathbb P_{\bm{\alpha}}$ be the $(\alpha_j)$-Bernoulli measure on $\{0, \ldots, J-1\}^\mathbb N$.

\begin{proof}[Proof of Theorem~\ref{t:main4}]
Fix a $y \in W(\bm{\alpha})$. Recall that the lower pointwise dimension of $m_{y,{\bm\alpha}}$ at the point $x \in [0,1]$ is defined by
$$\underline{d}_{m_{y,{\bm\alpha}}}(x) =\liminf_{r\to0}\frac{\log m_{y,{\bm\alpha}}(B(x,r))}{\log r},$$
 where $B(x,r)$ is the open interval of length $2r$ centred at $x$. One can verify that the collection $\{ \Delta_y(k_1 \cdots k_n) \, : \, n \in \mathbb N\}$ satisfies conditions (CB1)--(CB3) of the Moran-type construction from \cite[Section 15]{Pes97}. Moreover, for any intervals $\Delta_y(k_1 \cdots k_n)$, $\Delta_y(k_1 \cdots k_n, k_{n+1})$, it holds that
\[ \textrm{diam}(\Delta_y(k_1 \cdots k_n)) \le (\max_{e \in \mathcal E} q_e )^n,  \quad (\min_{e \in \mathcal E} q_e) \cdot\textrm{diam}(\Delta_y(k_1 \cdots k_n)) \le \textrm{diam}(\Delta_y(k_1 \cdots k_n k_{n+1})). \]
Therefore, by e.g.~\cite[Theorem~15.3(1)]{Pes97}, we can replace the balls $B(x,r)$ in the definition of $\underline{d}_{m_{y,{\bm\alpha}}}$ with the fibre fundamental intervals $ \Delta_y (k_1, \ldots, k_n)$ to obtain an upper bound for $\underline{d}_{m_{y,{\bm\alpha}}}(x)$ for all $x\in F_y(\bm{\alpha})$ and a lower bound for $m_{y,{\bm\alpha}}$-a.e.~$x\in[0,1]$ in the case that $m_{y,\bm{\alpha}}(F_y(\bm{\alpha}))=1$. To be more precise, for $x\in F_y(\bm{\alpha})$ with $\omega = ((j_\ell,k_\ell))_{\ell\ge 1} \in \pi^{-1}\{(x,y)\}$ that have $\tau_e(\omega)= \alpha_e$ for each $e \in \mathcal E$, we find that
	\begin{align*}
		\underline{d}_{m_{y,{\bm\alpha}}}(x)	&\leq\lim_{n\to\infty}\frac{\log m_{y,{\bm\alpha}}(\Delta_y(k_1\dotsb k_n))}{\log \textrm{diam}(\Delta_y(k_1\dotsb k_n))}\\
								&=\lim_{n\to\infty}\frac{\log\prod_{1\leq\ell\leq n}\frac{\alpha_{(j_\ell,k_\ell)}}{\alpha_{j_\ell}}}{\log\prod_{1\leq\ell\leq n}q_{(j_\ell,k_\ell)}}\\
								&=\lim_{n\to\infty}\frac{\frac1n\sum_{1\leq\ell\leq n}\log\alpha_{(j_\ell,k_\ell)}-\frac1n\sum_{1\leq\ell\leq n}\log\alpha_{j_\ell}}
				{\frac1n\sum_{1\leq\ell\leq n}\log q_{(j_\ell,k_\ell)}}.
	\end{align*}
	By collecting like terms, we find that
	\begin{align*}
		\sum_{1\leq\ell\leq n}\log\alpha_{(j_\ell,k_\ell)}	&=\sum_{e\in\mathcal E}\#\{1 \le i \le n \, : \, \omega_i=e\}\log\alpha_e,\\
		\sum_{1\leq\ell\leq n}\log\alpha_{j_\ell}	        &=\sum_{0 \le j < J}\#\{1 \le i \le n \, : \, j_i=j\}\log\alpha_j,\\
		\sum_{1\leq\ell\leq n}\log q_{(j_\ell,k_\ell)}	&=\sum_{e\in\mathcal E}\#\{1 \le i \le n \, : \, \omega_i=e\}\log q_e.
	\end{align*}
Since $x\in F_y({\bm\alpha})$, we have for each $e \in \mathcal E$ and $0 \le j < J$ that
\[ \begin{split}
\lim_{n\to\infty}\frac{\#\{1 \le i \le n \, : \, \omega_i=e\}}n=\ & \tau_e(\omega)=\alpha_e,\\
\lim_{n\to\infty}\frac{\#\{1 \le i \le n \, : \, j_i=j\}}n=\ & \lim_{n\to\infty}\sum_{0\leq k<B_j}\frac{\#\{1\leq i\leq n:\omega_{(j_i,k)}=(j,k)\}}n=\alpha_j.
\end{split}\]
Thus, recalling the definition of measure-theoretic entropy from \eqref{q:entropy}, we find that
 $$\underline{d}_{m_{y,{\bm\alpha}}}(x)\leq\frac{\sum_{e\in\mathcal E}\alpha_e\log\alpha_e-\sum_{0 \le j < J}\alpha_j\log\alpha_j}{\sum_{e\in\mathcal E}\alpha_e\log q_e}=\frac{h_{\mu_{\bm\alpha}}(\sigma)-h_{\mathbb P_{\bm\alpha}}(\sigma_J)}{-\sum_{e \in \mathcal E} \alpha_e \log q_e}$$
	for all $x\in F_y({\bm\alpha})$. Therefore, it follows from Lemma~\ref{l:pwdim}(i) that
$$\dim_{\mathrm{H}}(F_y({\bm\alpha}))\leq\frac{h_{\mu_{\bm\alpha}}(\sigma)-h_{\mathbb P_{\bm\alpha}}(\sigma_J)}{-\sum_{e \in \mathcal E} \alpha_e \log q_e}.$$

To prove the second statement, fix $y\in W({\bm\alpha})$ such that $m_{y,\bm{\alpha}}(F_y(\bm{\alpha}))=1$, which holds for $\nu_{\bm{\alpha}}$-a.e.~$y \in W(\bm{\alpha})$ by Lemma~\ref{Full}. Let $(j_m)_{m \ge 1} \in \pi_J^{-1}\{y\}$ be the unique sequence with $\tau_j((j_m)_{m \ge 1}) = \alpha_j$ for each $0 \le j < J$, see Lemma~\ref{l:uniquew}. By the above computations for the upper bound of $\dim_{\mathrm{H}}(F_y({\bm\alpha}))$ together with \cite[Theorem 15.3(2)]{Pes97}, we have for $m_{y,{\bm\alpha}}$-a.e.~$x\in [0,1]$ that
\begin{equation}\label{q:inflim}
\underline{d}_{m_{y,{\bm\alpha}}}(x)	\geq \inf \lim_{n\to\infty}\frac{\log m_{y,{\bm\alpha}}(\Delta_y(k_1\dotsb k_n))}{\log\textrm{diam}(\Delta_y(k_1\dotsb k_n))},
\end{equation}
where the infimum is taken over all sequences $(k_m)_{m \ge 1}$ such that $(j_m,k_m)_{k \ge 1} \in \pi^{-1}\{(x,y)\}$. We have seen that the set of $x$ for which $\#\pi^{-1}\{(x,y)\}>1$ is countable so is therefore a $m_{y,{\bm\alpha}}$-null set. Consequently, the infimum on the right-hand side of \eqref{q:inflim} is over a single sequence for $m_{y,{\bm\alpha}}$-a.e. $x\in[0,1]$. Fix $x\in F_y(\bm{\alpha})$ such that \eqref{q:inflim} and $\#\pi^{-1}\{(x,y)\}=1$ both hold. Then
\begin{align*}
		\underline{d}_{m_{y,{\bm\alpha}}}(x)	&\geq \lim_{n\to\infty}\frac{\log m_{y,{\bm\alpha}}(\Delta_y(k_1\dotsb k_n))}{\log\textrm{diam}(\Delta_y(k_1\dotsb k_n))}\\
	&=\lim_{n\to\infty}\frac{\sum_{e\in\mathcal E}\frac{\#\{1 \le i \le n \, : \, \omega_i=e\}}n\log\alpha_e-\sum_{0 \le j < J}\frac{\#\{1 \le i \le n \, : \, j_i=j\}}n \log\alpha_j}{\sum_{e\in\mathcal E}\frac{\#\{1 \le i \le n \, : \, \omega_i=e\}}n\log q_e}.
	\end{align*}
Since $x \in F_y(\bm{\alpha})$ for each $e \in \mathcal E$ and $0 \le j <J$, we find
\[\begin{split}
    \tau_e(\omega)=\ &\lim_{n\to\infty}\frac{\#\{1 \le i \le n \, : \, \omega_i=e\}}n=\alpha_e \, \, \text{ and }\\
    \tau_j((j_m)_{m \ge 1})=\ & \lim_{n\to\infty}\frac{\#\{1 \le i \le n \, : \, j_i=j\}}n=\alpha_j.
\end{split}\]
Therefore,
  $$\underline{d}_{m_{y,{\bm\alpha}}}(x)\geq\frac{\sum_{e\in\mathcal E}\alpha_e\log\alpha_e-\sum_{0 \le j < J}\alpha_j\log\alpha_j}{\sum_{e\in\mathcal E}\alpha_e\log q_e}=\frac{h_{\mu_{\bm\alpha}}(\sigma)-h_{\mathbb P_{\bm\alpha}}(\sigma_J)}{-\sum_{e \in \mathcal E} \alpha_e \log q_e}.$$
Since this holds for $m_{y,{\bm\alpha}}$-a.e.~$x \in F_y(\bm{\alpha})$ and $m_{y,{\bm\alpha}}(F_y({\bm\alpha}))=1$, it follows from Lemma~\ref{Full} and Lemma~\ref{l:pwdim}(ii) that
\[  \dim_{\mathrm{H}}(F_y({\bm\alpha}))   \geq\dim_{\mathrm{H}}(m_{y,{\bm\alpha}}) \geq\frac{h_{\mu_{\bm\alpha}}(\sigma)-h_{\mathbb P_{\bm\alpha}}(\sigma_J)}{-\sum_{e \in \mathcal E}\alpha_e \log q_e}. \qedhere\]
\end{proof}

\begin{remark}
The additional condition on the frequency vector $\bm{\alpha}$ that for each $0 \le j <J$ there are $k,\ell \in B_j$ with $\alpha_{(j,k)}>0$ and $\alpha_{(j,\ell)}>0$ is used only to remove the infimum in \eqref{q:inflim}. Another assumption that would allow us to remove the infimum is to assume that $\# \pi^{-1}\{(x,y)\}=1$ for all $(x,y) \in F(\bm{\alpha})$. This holds for example in the following cases.
\begin{itemize}
\item[(i)] If $B_j \ge 3$ for some $0 \le j < J$ and there is a $1 \le k\le B_j-2$ with $\alpha_{(j,k)}>0$, then for any $(x,y)$ with $\#\pi^{-1}\{(x,y)\} >1$ and any $\omega \in \#\pi^{-1}\{(x,y)\}$, we obtain $\tau_{(j,1)}(\omega)=0 \neq \alpha_{(j,k)}$ and thus $(x,y) \not \in F(\bm{\alpha})$.
\item[(ii)] If $\alpha_e >0$ for each $e \in \mathcal E$ and $h_{(j,0)}(0)=0$ and $h_{(j,B_j-1)}(1)=1$ for all $0 \le j < J$, then for any $(x,y)$ with $\#\pi^{-1}\{(x,y)\} >1$ any $\omega \in \#\pi^{-1}\{(x,y)\}$ will either end in an infinite string of digits from the set $\{ (j,0) \, : \, 0 \le j < J\}$ or in an infinite string of digits from the set $\{ (j,B_j-1) \, : \, 0 \le j < J\}$ and again there is at least one $e \in \mathcal E$ for which $\tau_e (\omega) = 0 \neq \alpha_e$.
\end{itemize}
\end{remark}

\section{Examples}\label{s:examples}

\begin{ex}
For a concrete example, let $J=2$, $B_0=2$, $B_1=3$ so that $$\mathcal E = \{ (0,0), (0,1), (1,0), (1,1), (1,2)\}.$$ Let
\[ \begin{split}
h_{(0,k)}(x) =\ & \frac{x+k}{2},  \quad k=0,1,\\
h_{(1,k)}(x) =\ & \frac{x+k}{3}, \quad k=0,1,2,
\end{split}\]
so $\varepsilon_e = 0$ for all $e \in \mathcal E$ and $r_{(0,1)}= \frac12$, $r_{(1,1)}=\frac13$ and $r_{(1,2)}= \frac23$. Take $p \in (0,1)$ arbitrary and let $p_0=p$, so $f_0(y) = py$ and $f_1(y) = (1-p)y+p$. This gives
\[ A_{(0,k)} = \left[\begin{array}{cc}
1/2 & 0 \\
0 & p
\end{array}\right], \, A_{(1,k)} = \left[\begin{array}{cc}
1/3 & 0 \\
0 & 1-p
\end{array}\right]\]
and 
\[ \bm{v}_{(0,0)} =  \bm{v}_{(1,0)} = \left[\begin{array}{c}
0 \\
0\end{array} \right], \, \bm{v}_{(0,1)} = \left[\begin{array}{c}
1/2 \\
0\end{array} \right], \, \bm{v}_{(1,1)}  = \left[\begin{array}{c}
1/3 \\
p\end{array} \right] , \, \bm{v}_{(1,2)}  = \left[\begin{array}{c}
2/3 \\
p\end{array} \right].\]
See Figure~\ref{f:2and3}(b) for an illustration of how this GLS IFS $\{ A_e + \bm{v}_e\}_{e \in \mathcal E}$ acts on $[0,1]^2$. For the number expansions, if $(e_m)_{m \ge 1} \in \mathcal E^\mathbb N$, then for each $m \ge 1$ we get $s_m=0$, $K_m=2$ if $j_m=0$ and $K_m=3$ if $j_m=1$ and $t_m \in \{0,\frac12, \frac13, \frac23\}$ for all $m \ge 1$. So, in fact, for each $(e_m)_{m \ge 1} \in \mathcal E^\mathbb N$ if we set $\kappa(n) = \#\{1 \le m \le n \, : \, j_m = 0\}$, then \eqref{q:glsifs} becomes
\[  \lim_{m \to \infty} h_{e_1} \circ \cdots \circ h_{e_m} (0) = \sum_{m \ge 1} \frac{t_m}{2^{\kappa(m)}3^{m-\kappa(m)}}.\]
Hence, this GLS IFS produces for each $x \in [0,1]$ number expansions in mixed base 2 and 3. Note that for this IFS, $h_e \neq h_{e'}$ if $e \neq e'$. So from Proposition~\ref{p:uncountable}, it follows that each $x \in [0,1]$ has uncountably many different expansions with mixed bases 2 and 3. If ${\bm \alpha}$ satisfies the assumption of Theorem \ref{t:main4}, then  we can apply Theorem~\ref{t:main4}. We have $\alpha_0 = \alpha_{(0,0)}+ \alpha_{(0,1)}$ and $\alpha_1 = \alpha_{(1,0)}+ \alpha_{(1,1)}+ \alpha_{(1,2)}$ and obtain
for $\nu_{\bm{\alpha}}$-a.e.~$y \in W(\bm{\alpha})$ that
\[ \dim_{\mathrm H}(F_y(\bm{\alpha})) = \frac{\sum_{e \in \mathcal E} \alpha_e \log \alpha_e - \alpha_0\log \alpha_0 - \alpha_1 \log \alpha_1}{-\alpha_0\log 2 - \alpha_1 \log 3}. \]

\medskip
Note that if we consider the IFS $\{ A_{(0,0)} + \bm{v}_{(0,0)}, A_{(1,2)} + \bm{v}_{(1,2)} \}$ with $A_{(j,k)}$ and $\bm{v}_{(j,k)}$ as in the example and we take $\frac12 < p < \frac23$, then we obtain a diagonally affine IFS that satisfies condition (D')(a). Hence, we can apply Theorem~\ref{t:main1} to this IFS to obtain for each $\alpha \in (0,1)$ an expression for the Hausdorff dimension of the set of points $(x,y) \in [0,1]^2$ that have a GLS expansion containing only the digits $(0,0)$ and $(1,2)$ and in which $(0,0)$ occurs with frequency $\alpha$ (and thus $(1,2)$ with frequency $1-\alpha$). Of course, here we can take any other combination of a digit from $\{(0,0), (0,1)\}$ and a digit from $\{(1,0), (1,1), (1,2)\}$ to obtain a similar result.

\medskip
We can extend this example in the following sense. Fix some $J \in \mathbb N_{\ge 2}$ and different integers $M_0,M_1, \ldots, M_{J-1} > J$. Also fix some probability vector $(p_j)_{0 \le j <J}$. So,
\[ \mathcal E = \{ (j,k) \, : \, 0 \le j < J , \, \, 0 \le k < M_j\}. \]
For $(j,k) \in \mathcal E$, set
\[ A_{(j,k)} = \left[\begin{array}{cc}
1/M_j & 0 \\
0 & p_j
\end{array}\right], \quad \bm{v}_{(j,k)} = \left[\begin{array}{c}
k/M_j \\
\sum_{i=0}^{j-1}p_i\end{array} \right].\]
For each $x \in [0,1]$ and sequence $(j_m)_{m \ge 1} \in \{0,1, \ldots, J-1\}^\mathbb N$, the GLS expansion produced by this system has the form
\begin{equation}\label{q:NM}
x = \sum_{m \ge 1} \frac{d_m}{M_0^{c_{0,m}} M_1^{c_{1,m}} \cdots M_{J -1}^{c_{J-1,m}}},
\end{equation}
with $d_m \in \{ 0, \ldots, M_{j_m}-1\}$ and $c_{j,m} = \# \{ 1 \le i \le m \, : \, j_i=j\}$. In other words, the system produces for each $x\in [0,1]$ uncountably many different mixed base expansions with bases $M_0, \ldots, M_{J-1}$. Here, we need to remark that we consider two GLS expansions  produced by the system different if the two corresponding sequences in
\[ \mathcal A = \bigcup_{(j,k) \in \mathcal E}\left\{ \left(0,M_j,\frac{k}{M_j}\right) \right\}  \]
are different. For the point $0$, for example, this means that the GLS expansions generated by the system are all of the expansions of the form
\[ 0=\sum_{m \ge 1} \frac{0}{M_0^{c_{0,m}} \cdots M_{J -1}^{c_{J-1,m}}},\]
with $(c_{0,m}, \ldots, c_{J-1,m}) \in \mathbb N^{J}$ satisfying $\sum_{\ell=0}^{J-1} c_{\ell,m} =m$.
\end{ex}

\begin{ex}
Fix an $N \in \mathbb N_{\ge 3}$ and a $0 < p < 1$ and let $\mathcal E = \{(j,k)\, : \, j=0,1, \, 0 \le k < N\}$. For $0 \le k <N$, set
\[ A_{(0,k)} = \left[\begin{array}{cc}
1/N & 0 \\
0 & p
\end{array}\right], \, A_{(1,k)} = \left[\begin{array}{cc}
-1/N & 0 \\
0 & 1-p
\end{array}\right]\]
and
\[ \bm{v}_{(0,k)} =  \left[\begin{array}{c}
k/N \\
0\end{array} \right], \, \bm{v}_{(1,k)} = \left[\begin{array}{c}
(k+1)/N \\
p\end{array} \right].\]
Then for any $x \in [0,1]$ and any $(j_m)_{m \ge 1} \in \{0,1\}^\mathbb N$, the number expansion of $x$ produced by this system has the form
\[ x = \sum_{m \ge 1} (-1)^{j_m} \frac{d_m}{N^m},\]
for some $d_m \in \{0, \ldots, N-1\}$, $m \ge 1$. So, the system produces for each $x$ a signed base $N$-expansion in which the signs of the terms correspond to a preset sequence of signs $(j_m)_{m \ge 1}$.

\medskip
Also this system satisfies $h_e \neq h_{e'}$ whenever $e \neq e'$ and together with any frequency vector $\bm{\alpha} \in (0,1)^{2N}$ for which the conditions of Theorem~\ref{t:main4} are satisfied, the Hausdorff dimension of the Besicovitch-Eggleston set $F_y(\bm{\alpha})$ for $\nu_{\bm{\alpha}}$-a.e.~$y \in W(\bm{\alpha})$ is given by Theorem~\ref{t:main4}. For $\frac1N < p < \frac{N-1}{N}$ and any $0 \le k,\ell <N$ the system $\{ A_{(0,k)} + \bm{v}_{(0,k)}, A_{(1,\ell)} + \bm{v}_{(1,\ell)} \}$ satisfies (D')(a), so then also Theorem~\ref{t:main1} applies.
\end{ex}

\def\cprime{$'$}

\end{document}